\documentclass[10pt]{article}

\usepackage{tikz}
\usetikzlibrary{decorations.pathreplacing, arrows,positioning}

\usepackage{amsmath,amssymb,amsfonts,amsthm,amsbsy,mathtools}
\usepackage{graphicx,xcolor}
\usepackage{hyperref}
\hypersetup{
    colorlinks=true,
    linkcolor=blue,
    citecolor=darkgray
}
\usepackage[T1]{fontenc}
\usepackage{enumitem}
\usepackage{footnote,pifont,dsfont}
\usepackage[mathscr]{eucal}
\usepackage{fullpage}
\usepackage{thmtools}
\usepackage{lipsum}
\usepackage{algorithmic}

\newtheorem{theorem}{Theorem}

\newtheorem{proposition}[theorem]{Proposition}
\newtheorem{lemma}[theorem]{Lemma}
\newtheorem{example}[theorem]{Example}

\theoremstyle{definition}
\newtheorem{definition}[theorem]{Definition}

\newtheorem{remark}[theorem]{Remark}

\date{}

\title{Higher Dimensional Lattice Walks: Connecting Combinatorial and Analytic Behavior}

\author{Stephen Melczer\thanks{University of Pennsylvania, Department of Mathematics, 209 S. 33rd Street, Philadelphia, PA 19104, (smelczer@sas.upenn.edu), \url{https://www.math.upenn.edu/\string~smelczer/}).} \and Mark C. Wilson\thanks{Department of Computer Science, University of Auckland, Private Bag 92019 Auckland, New Zealand (mc.wilson@auckland.ac.nz), \url{http://mcw.blogs.auckland.ac.nz/}).} }

\newcommand{\mD}{\ensuremath{\mathcal{D}}}
\newcommand{\mC}{\ensuremath{\mathcal{C}}}

\newcommand{\mG}{\ensuremath{\mathcal{G}}}

\newcommand{\mR}{\ensuremath{\mathcal{R}}}
\newcommand{\mS}{\ensuremath{\mathcal{S}}}

\newcommand{\mV}{\ensuremath{\mathcal{V}}}

\newcommand{\bi}{\ensuremath{\mathbf{i}}}

\newcommand{\bo}{\ensuremath{\mathbf{1}}}

\newcommand{\br}{\ensuremath{\mathbf{r}}}

\newcommand{\bw}{\ensuremath{\mathbf{w}}}

\newcommand{\by}{\ensuremath{\mathbf{y}}}
\newcommand{\bz}{\ensuremath{\mathbf{z}}}

\newcommand{\oS}{\ensuremath{\overline{S}}}

\newcommand{\ox}{\ensuremath{\overline{x}}}
\newcommand{\oy}{\ensuremath{\overline{y}}}
\newcommand{\oz}{\ensuremath{\overline{z}}}

\newcommand{\bzhat}{\bz_{\hat{k}}}
\newcommand{\bzht}[1]{\bz_{\hat{#1}}}

\newcommand{\bt}{\ensuremath{\boldsymbol \theta}}

\newcommand{\bzer}{\ensuremath{\mathbf{0}}}
\newcommand{\bone}{\ensuremath{\mathbf{1}}}

\newcommand{\sgn}{\operatorname{sgn}}

\newcommand{\btht}[1]{\bt_{\hat{#1}}}

\newcommand{\bp}{\ensuremath{\mathbf{p}}}
\newcommand{\bpht}[1]{\bp_{\hat{#1}}}

\newcommand{\St}{\tilde{S}}

\DeclareMathOperator{\gradlog}{\nabla_{\log}}

\newcommand{\diagr}[1]{
  \begin{tikzpicture}[scale=0.2]\makediagb{#1}\end{tikzpicture}
}

\makeatletter
\def\testbb#1{\testbb@i#1,,\@nil}%
\def\testbb@i#1,#2,#3\@nil{%
  \draw (O) --++(#1);
  \ifx\relax#2\relax\else\testbb@i#2,#3\@nil\fi}
\makeatother   

\newcommand{\makediagb}[1]{
    \coordinate (O) at (0,0); \coordinate (N) at (0,1);
    \coordinate (NE) at (1,1); \coordinate (E) at (1,0);
    \coordinate (SE) at (1,-1); \coordinate (S) at (0,-1);
    \coordinate (SW) at (-1,-1);\coordinate (W) at (-1,0);
    \coordinate (NW) at (-1,1); \coordinate (B1) at (1.2,1.2);
    \coordinate (B2) at (-1.2,-1.2);
    
    \draw (B1) --++(0,-2.4); \draw (B1) --++ (-2.4,0);
    \draw (B2) --++(0,2.4);  \draw (B2) --++ (2.4,0);   
    \testbb{#1}
}

\begin{document}
\maketitle

\begin{abstract}
We consider the enumeration of walks on the non-negative lattice $\mathbb{N}^d$, with steps defined by a set $\mS \subset \{-1, 0, 1\}^d \setminus \{\bzer\}$. Previous work in this area has established asymptotics for the number of walks in certain families of models by applying the techniques of analytic combinatorics in several variables (ACSV), where one encodes the generating function of a lattice path model as the diagonal of a multivariate rational function. Melczer and Mishna obtained asymptotics when the set of steps $\mS$ is symmetric over every axis; in this setting one can always apply the methods of ACSV to a multivariate rational function whose  set of singularities is a smooth manifold (the simplest case). Here we go further, providing asymptotics for models with generating functions that must be encoded by multivariate rational functions having non-smooth singular sets.  In the process, our analysis connects past work to deeper structural results in the theory of analytic combinatorics in several variables.  One application is a closed form for asymptotics of models defined by step sets that are symmetric over all but one axis. As a special case, we apply our results when $d=2$ to give a rigorous proof of asymptotics conjectured by Bostan and Kauers; asymptotics for walks returning to boundary axes and the origin are also given. 
\end{abstract}



\section{Introduction}
\label{sec:intro}

Much modern research in enumeration concerns links between analytic function behaviour and combinatorial models.  When one has sufficient information about a generating function --- for instance, the locations and types of its singularities closest to the origin in the complex plane --- the theory of analytic combinatorics in one variable gives, almost automatically in many cases, the asymptotics of the related sequence (see the compendium text of Flajolet and Sedgewick~\cite{ FlSe09} for further details).  

When a generating function is given in terms of an expansion of a multivariate power series, however, much less is known.  Over the last two decades, several authors have been working towards the development of a theory of analytic combinatorics in several variables. This refers to methods for deriving asymptotics of a sequence
\[ b_n = [\bz^{n\br}]F(\bz) = \left[z_1^{n\cdot r_1} \cdots z_d^{n\cdot r_d}\right] F(z_1,\dots,z_d) := a_{n\, r_1, \dots , n\, r_d}\]
for some fixed vector $\br \in \mathbb{Z}^d$ and multi-dimensional sequence $(a_\bi)_{\bi \in \mathbb{Z}^d}$ such that 
\[ F(\bz) :=  \sum_{\bi \in \mathbb{Z}^d} a_\bi \bz^{\bi} \]
is a meromorphic function analytic in a specified domain. The theory, as it has developed, generally consists of two stages: first, one must find the \emph{contributing} singularities of $F(\bz)$, which are the singularities of $F(\bz)$ where local behaviour of the function dictates its coefficients' asymptotics.  Once these contributing points have been found, which is often the most difficult step of the analysis, one must then calculate the \emph{asymptotic contribution} of each point and sum the results.

In 2002, Pemantle and Wilson~\cite{PeWi02} derived asymptotics for the power series expansion of a rational function $F(\bz) = G(\bz)/H(\bz)$ admitting a finite number of contributing points, at which the variety of singularities is a complex manifold.  Two years later~\cite{PeWi04}, they extended this analysis to allow contributing points where the variety of singularities is the transverse union of complex manifolds.  Modern approaches incorporate techniques from differential and algebraic geometry, topology, and singularity theory; the interested reader is referred to the text of Pemantle and Wilson~\cite{PeWi13} or, for a more elementary introduction, the thesis of Melczer~\cite{Melczer2017}.

\subsection{Lattice walks in restricted regions}

The enumeration of lattice walks restricted to certain regions is a classical topic in combinatorics, tracing its roots back hundreds of years to work on what is now known as the ballot problem. In modern times, a large area of work on this topic centers around using the so-called \emph{kernel method} to express generating functions of large classes of models as positive series extractions or diagonals of multivariate rational functions, proving that the generating functions are (or are not) \emph{D-finite}; that is, determining whether they satisfy a linear differential equation with polynomial coefficients. Although there are a finite number of models with steps in $\{\pm1,0\}^2$ that are restricted to a quarter-plane, their combinatorics has been the subject of intense study in recent years~\cite{FaIaMa99,Bous05,KaKoZe09,Mish09,BoMi10,BoKa09,BoKa10,Rasc12,MiRe09,MeMi14a,MelczerMishna2016,BostanBousquet-MelouKauersMelczer2016,Bousquet-Melou2016,BernardiBousquet-MelouRaschel2016,BostanChyzakHoeijKauersPech2017,DreyfusHardouinRoquesSinger2018}. The models considered in this paper are higher-dimensional generalizations of this two-dimensional setting.

A lattice path model is determined by a finite set of steps $\mS\subset\mathbb{Z}^d$ and a region $\mathcal{P} \subset \mathbb{Z}^d$ to which the walks of the model are restricted.  A model whose set of allowed steps is symmetric over every axis is called \emph{highly symmetric}. Melczer and Mishna~\cite{MelczerMishna2016} combined the kernel method with techniques from analytic combinatorics in several variables to give asymptotics for the number of integer lattice walks restricted to an orthant (that is, $\mathcal{P} = \mathbb{N}^d$) when the set of steps $\mS$ is a subset of $\{\pm1,0\}^d \setminus \{\bzer\}$ and is highly symmetric.   The main result of their paper is the following.

\begin{theorem}[{Melczer and Mishna~\cite[Theorem 3.4]{MelczerMishna2016}}] \label{thm:MeMiAsm}
  Let $\mS \subset \{-1,0,1\}^d \setminus \{\mathbf{0}\}$ be a set of 
  steps that is symmetric over every axis and moves forwards and backwards in each coordinate.
  Then the number~$s_n$ of walks of length~$n$ taking steps in~$\mS$,
  beginning at the origin, and never leaving the positive orthant has
  asymptotic expansion
\begin{equation*}
s_n  =|\mS|^n \cdot n^{-d/2} \cdot  \left(\, \left(s^{(1)} \cdots s^{(d)}\right)^{-1/2} \pi^{-d/2} |\mS|^{d/2}+ O\left( \frac{1}{n} \right)\, \right),
\end{equation*}
where $s^{(k)}$ denotes the number of steps in $\mS$ that have $k^\text{th}$ coordinate $1$.
\end{theorem}
In order to get this result, Melczer and Mishna used the kernel method to derive the expression 
\[ F(t) = \Delta \left(\frac{G(\bz,t)}{H(\bz,t)}\right) = \Delta\left(\frac{(1+z_1)\cdots(1+z_d)}{1 - t(z_1\cdots z_d)S(\bz)}\right) \]
for the generating function $F(t)$ counting the number of walks of a given length defined by the model, where $\Delta$ is the diagonal operator defined in Section~\ref{sec:diag} below. After verification of certain technical conditions, Theorem~\ref{thm:MeMiAsm} then follows from the main result of the original paper of Pemantle and Wilson~\cite{PeWi02}. The symmetry condition on the step set $\mS$ was chosen by Melczer and Mishna precisely because it leads to a smooth \emph{singular variety}, the set of singular points of $G/H$, defined by the zero set of $H$. 

Although not originally considered by Melczer and Mishna, Theorem~\ref{thm:MeMiAsm} can be easily extended~\cite[Chapter 7]{Melczer2017} to handle positively weighted step sets with weights that are symmetric over every axis. Generalizations of this work have used multivariate singularity analysis to enumerate weighted walks~\cite{CourtielMelczerMishnaRaschel2017} and walks with step coordinates having absolute value greater than one~\cite{BostanBousquet-MelouMelczer2018}. 

\subsection{Our contributions}
Building on the conference paper of Melczer and Wilson~\cite{MelczerWilson2016}, in this work we generalize the results of Melczer and Mishna by giving asymptotics of lattice path models, restricted to the positive orthant, whose set of allowable steps is symmetric over all but one axis.  The formulae we obtain give explicit and fairly simple descriptions of the exponential rate, leading order and leading coefficient in terms of the basic data of the walk step set $\mS$. The \emph{drift} of a walk, the vector sum of all the steps in $\mS$, plays a crucial role in asymptotics; owing to symmetry, only one coordinate of the drift may be non-zero and we refer to a walk as having negative, zero, or positive drift depending on the sign of this coordinate.

Additionally, we show how the arguments of Melczer and Mishna's analysis fit into larger structure theorems about the singularities of multivariate rational functions and their asymptotic expansions. In the negative and zero drift cases, substantial extra work is needed because the expected leading order coefficient for generic problems of this dimension turns out to vanish. Unfortunately, due to this vanishing and other degeneracies in integrals that must be asymptotically approximated, we are currently unable to determine asymptotics for the general zero drift case. Such models will be the subject of future study.

Furthermore, we provide the first rigorous proofs of the guessed asymptotics of Bostan and Kauers~\cite{BoKa09} on 2D walks restricted to the non-negative quadrant, completing the outline in Melczer and Wilson~\cite{MelczerWilson2016}. Our analysis uncovers and explains periodicity of the asymptotic coefficients in the negative drift case, which was not noted in \cite{BoKa09}, and we give asymptotics for the number of walks returning to the boundary axes and the origin. Around the same time as the conference paper of Melczer and Wilson~\cite{MelczerWilson2016}, Bostan et al.~\cite{BostanChyzakHoeijKauersPech2017} rigorously gave annihilating differential equations for the generating functions of these lattice path models. Using these differential equations they were able to prove some of the guessed asymptotics of Bostan and Kauers~\cite{BoKa09}, however due to issues related to the decidability of asymptotics for coefficients of D-finite functions they were unable to prove all asymptotics. We discuss the difficulties they faced, and how our results fit into this context, in Section~\ref{sec:computability}. An accompanying Maple worksheet verifying our calculations can be found online\footnote{\url{https://github.com/smelczer/HigherDimensionalLatticeWalks}}.

\subsection{Past work}

Lattice path models restricted to a halfspace have algebraic generating functions that can be explicitly determined~\cite{Gessel1980}, leading to strong asymptotic results~\cite{BaFl02}. For this reason, much attention has been devoted to walks in quadrants and related cones such as orthants. Early combinatorial works in this area include Kreweras~\cite{Krew65} and Gessel~\cite{Gessel1986}; in 2005, Bousquet-M{\'e}lou~\cite{Bous05} introduced the \emph{algebraic kernel method}, on which our formal series setup is heavily based, to study a quadrant lattice path model stemming from the work of Kreweras. Gessel and Zeilberger~\cite{GeZe92} gave representations for lattice path generating functions in so-called Weyl chambers in arbitrary dimension, which are equivalent to the diagonal representations of Melczer and Mishna~\cite{MelczerMishna2016} for the Weyl chamber $A_1^d$. This has been a fruitful area of research: see also Zeilberger~\cite{Zeilberger1983}, Grabiner and Magyar~\cite{GrMa93}, Tate and Zelditch~\cite{TateZelditch2004}, and Feierl~\cite{Feierl2014,Feierl2018}. The systematic combinatorial enumeration of walks in a quadrant was popularized by Bousquet-M{\'e}lou and Mishna~\cite{BoMi10}, following work of Petkov\v{s}ek~\cite{Petkovsek1998}, Bousquet-M{\'e}lou and Petkov\v{s}ek~\cite{Bousquet-MelouPetkovsek2003}, and Mishna~\cite{Mishna2007}, among others.

Walks in quadrants and orthants have also been long studied from a probabilistic perspective. In one approach, developed in part for problems arising in queuing theory, a singularity analysis of solutions to functional equations satisfied by lattice path generating functions yields analytic and asymptotic information. The text of Fayolle et al.~\cite{FaIaMa99} gives a detailed view on the techniques involved, some of which inspired Bousquet-M{\'e}lou's creation of the algebraic kernel method; see also Maly\v{s}ev~\cite{Malysev1971} for an early history. The lattice path models we study have asymptotics of the form $C \, n^{\alpha} \, \rho^n$ for constants $\alpha$ and $\rho$, where $C$ is constant or depends only on the periodicity of $n$.  Fayolle and Raschel~\cite{FaRa12} used these techniques to outline a method that, in principle, allows one to calculate the exponential growth $\rho$ for many quadrant models.

Another probabilistic approach to lattice path enumeration is to use local limit theorems and/or approximate discrete walks by scaling limits such as multidimensional Brownian motion. For a large variety of step sets and restricting cones, including orthants, Denisov and Wachtel~\cite{DenisovWachtel2015} give techniques for determining the exponential growth $\rho$ and exponent $\alpha$ for the number of walks that begin and end at the origin: their asymptotic formulas are given by an explicit expression involving the smallest positive eigenvalue of the Laplace-Beltrami operator on a sphere. When the step set under consideration has zero drift, these techniques also give $\rho$ and $\alpha$ for the number of walks ending anywhere in the restricting cone. Among other results, Duraj~\cite[Example~7]{Du14} determines $\rho$ and $\alpha$ for quadrant walks ending anywhere in the quadrant when the drift has negative coordinates, and Garbit and Raschel~\cite{GarbitRaschel2016} give the exponential growth $\rho$ for walks ending anywhere in the restricting cone under no restriction on the drift of a model. We note that using these probabilistic techniques it is very difficult, if not impossible, to determine the leading asymptotic constant $C$ or to determine higher order asymptotic terms, as our approach provides (in a less general setting).

Finally, lattice path enumeration has been studied through the lens of computer algebra. Among the many results in this area we mention: Kauers et al.~\cite{KaKoZe09}, which proved a longstanding open problem on the enumeration of certain quarter plane walks using creative telescoping techniques; Bostan and Kauers~\cite{BoKa09}, which computationally guessed certain differential equations satisfied by the generating functions of quarter plane models and used this to guess the asymptotics we prove in Section~\ref{sec:2DWalks}; and Bostan et al.~\cite{BoRaSa14}, which used the work of Denisov and Wachtel mentioned above to create algorithms explicitly determining the exponential growth $\rho$ and exponent $\alpha$ for quarter plane models.

The enumeration of lattice walks is a thriving area of enumerative combinatorics, with too many results to explicitly mention here. Those looking for additional resources can investigate the texts of Mohanty~\cite{Mohanty1979} and Narayana~\cite{Narayana1979}, and the survey of Krattenthaler~\cite[Chapter 10]{Bona2015}.

\subsection{Organization}

We begin in Section~\ref{sec:Results} by discussing our main results and some illustrative examples. Section~\ref{sec:KernelMethod} then gives an overview of the kernel method applied to lattice path enumeration, and shows how it can be used to derive expressions for lattice path generating functions that are amenable to the techniques of analytic combinatorics in several variables.  Unlike the previous work of Melczer and Mishna --- where complete symmetry of the step sets under consideration simplified the required manipulations of the kernel method --- care must be taken here when manipulating diagonal and positive sub-series extractions in iterated Laurent series rings.  Section~\ref{sec:contrib}  details the general methods of analytic combinatorics in several variables, and outlines how the asymptotic analysis will proceed; in Proposition~\ref{prop:minpts} we give an explicit description of the contributing singularities for the models under consideration.  We derive asymptotics using this characterization in Section~\ref{sec:Asym}. This work divides naturally into three cases: the positive, negative and zero drift, listed in increasing order of difficulty (as mentioned above, we do not treat the general zero drift case here; see Section~\ref{sec:zerodrift} for more information). Detailed computations needed are collected in Appendix~\ref{appendixA}. Section~\ref{sec:2DWalks} proves asymptotics of 2D walks restricted to the non-negative quadrant.  Section~\ref{sec:Conclusion} discusses extensions and directions for future research.

A summary of results is displayed in Table~\ref{tab:summary}. It has previously been observed~\cite{FaRa12,CourtielMelczerMishnaRaschel2017} that the sub-exponential order term $n^\alpha$ appearing in asymptotics for the number of walks in a lattice path model has some correlation with the drift of its steps. This phenomenon occurs here: each coordinate where the drift is negative corresponds to a contribution of $n^{-3/2}$ to dominant asymptotics, a positive drift coordinate does not effect the order term, and a zero drift coordinate corresponds to an asymptotic contribution of at most $n^{-1/2}$ (depending on whether or not the walk is highly symmetric).

\begin{table}[h]
\centering
\begin{tabular}{|llllll|}
\hline
Drift & Exists for & Exp. rate & Order & Geometry & Theorem\\ \hline
positive & $d\geq 2$ & $|S|$ & $n^{-(d-1)/2}$ & nonsmooth & Theorem~\ref{thm:PosAsm} \\
negative & $d\geq 2$ &$< |S|$  & $n^{-1-d/2}$ & smooth & Theorem~\ref{thm:NegAsm} \\
zero (h.s.) &$d\geq 2$ &$|S|$ &$n^{-d/2}$ & smooth & \cite[Thm 3.4]{MelczerMishna2016}\\
zero (not h.s.) & $d\geq 3$ & $|S|$ & {\footnotesize$O(n^{1/2 - d/2})$} & nonsmooth & Conjectural \\
\hline
\end{tabular}
\caption{Summary of results: exponential rate and leading asymptotic order. Here ``h.s." means ``highly symmetric" (step set is symmetric over every axis).} \label{tab:summary}
\end{table}

\section{Main Results and Examples}
\label{sec:Results}
In order to simplify equations, we fix a dimension $d\in\mathbb{N}$ and use the following notation multi-index:
\[ \oz_i= z_i^{-1}; \qquad \bz = (z_1, \dots, z_d); \qquad \bi =( i_1,i_2,\dots,i_d)\in \mathbb{Z}^d; \]
\[ \bz^{\bi} = z_1^{i_1}\cdots z_d^{i_d}; \qquad \bz_{\hat{k}} := (z_1,\dots,z_{k-1},z_{k+1},\dots,z_d). \]

We consider walks in dimension $d$ defined by a (finite) set $\mS \subset \{\pm1,0\}^d \setminus \{\bzer\}$ of weighted steps, where $\bi \in \mS$ is given real weight $w_{\bi} \geq 0$, such that
\begin{itemize}
\item there exists some step forwards and some step backwards in the direction of each coordinate axis:
\[ \text{For all $j=1,\dots,d$ there exists $\bi \in \mS$ with $\bi_j = 1$ and $w_{\bi} \neq 0$}; \] 
\[ \text{For all $j=1,\dots,d$ there exists $\bi \in \mS$ with $\bi_j = -1$ and $w_{\bi} \neq 0$}. \] 
\item the weighting $w_{\bi}$ is symmetric over all axes except one;
\item each walk is confined to the non-negative orthant $\mathbb{N}^d$.
\end{itemize}

The \emph{(weighted) characteristic polynomial} of $\mS$ is the Laurent polynomial
\[ S(\bz) = \sum_{\bi \in \mS} w_{\bi} \bz^\bi. \]
We may assume without loss of generality that the axis of non-symmetry is $z_d$.  In other words, our step set $\mS$ is such that for each $j$ with $1\leq j \leq d-1$,
\[ S(z_1,\dots,z_{j-1},\oz_j,z_{j+1}, \dots, z_d) = S(\bz), \]
so we may write
\[ S(\bz) = \oz_d A\left(\bzht{d}\right) + Q\left(\bzht{d}\right) + z_d B\left(\bzht{d}\right)   \]
for Laurent polynomials $A,B,$ and $Q$ that are symmetric in their variables. Note that $S(\bone) = |\mS|$ is the size of the step set when each step has weight $1$.

Also important is the \emph{drift} $B(\bone) - A(\bone)$ of a walk with respect to the $d$-axis, the weight of steps in the positive $z_d$ direction minus the weight of steps in the negative $z_d$ direction. For our models the sign of the drift will correspond to different asymptotic regimes; in general the relationship between drift and asymptotics is more nuanced~\cite[Section 6.4]{CourtielMelczerMishnaRaschel2017}. Although the models we consider start at the origin, it is possible to modify our approach with minimal overhead to start at any $\bi \in \mathbb{N}^d$. In fact, one can treat the starting point $\bi$ as a parameter that will appear only in the leading constant of dominant asymptotics for the number of walks, and this approach can be used to construct discrete harmonic functions~\cite{CourtielMelczerMishnaRaschel2017}.

\subsection{Positive drift models}
\label{ss:pos}
For $1\leq k \leq d-1$, define $b_k$ to be the total weight of the steps moving forwards (or backwards) in the $k$th coordinate,
\[ b_k = \sum_{\bi \in \mS, i_k=1} w_{\bi} = \sum_{\bi \in \mS, i_k=-1} w_{\bi}. \]
Our main asymptotic result for positive drift models is the following. 

\begin{theorem}[Positive Drift Asymptotics]
\label{thm:PosAsm} 
Let $\mS$ be a (weighted) step set that is symmetric over all but one axis and takes a step forwards and backwards in each coordinate.  If $\mS$ has positive drift, then then number of walks of length $n$ that never leave the non-negative orthant satisfies
\begin{equation} s_n = S(\bone)^n \cdot n^{\frac{-(d-1)}{2}} \cdot \left[ \left(1 - \frac{A(\bone)}{B(\bone)}\right) \left(\frac{S(\bone)}{\pi}\right)^{\frac{d-1}{2}} \frac{1}{\sqrt{b_1 \cdots b_{d-1}}} \right]\left(1 + O(n^{-1})\right). \label{eqn:PosAsm} \end{equation}
\end{theorem}

Note that the result is trivial to apply to any given model, and is general enough to handle families of models in varying dimension.

\begin{example}
\label{eg:posdrift}
Consider the step set where $B\left(\bzht{d}\right) = \prod_{j<d} (z_j + \oz_j)$, $Q\left(\bzht{d}\right)=0$, and $A\left(\bzht{d}\right)=1$, so each step has weight $1$. Then
\[ s_n = \left(1+2^{d-1}\right)^n \cdot n^{\frac{-(d-1)}{2}} \cdot \left[ \frac{2^{d-1} - 1}{(2^d\pi)^\frac{d-1}{2}} \right]\left(1 + O(n^{-1})\right). \]
\end{example}

Theorem~\ref{thm:PosAsm} is proved in Section~\ref{ss:posdrift}.

\subsection{Negative drift models}
\label{ss:neg}
Dominant asymptotics in the negative drift case are given by adding the asymptotic contributions of a finite collection of points. Let $\rho = \sqrt{\frac{A(\bone)}{B(\bone)}}$, and for each $1 \leq k \leq d-1$ define
\[ b_k(\bzht{k}) := [z_k]S(\bz) = [z_k^{-1}]S(\bz). \]
Furthermore, define
\[ C_{\rho} := \frac{S(\bone,\rho) \, \rho}{2\, \pi^{d/2}\, A(\bone) (1-1/\rho)^2} \cdot \sqrt{\frac{S(\bone,\rho)^d}{\rho \, b_1(\bone,\rho) \cdots b_{d-1}(\bone,\rho) \cdot B(\bone)}}\]
and let $C_{-\rho}$ be the constant obtained by replacing $\rho$ by $-\rho$ in $C_{\rho}$ (the term in the square-root will always be real and positive, so there is no ambiguity).

\begin{theorem}[Negative Drift Asymptotics]
\label{thm:NegAsm}
Let $\mS$ be a negative drift (weighted) step set that is symmetric over all but one axis and takes a step forwards and backwards in each coordinate.  If $Q(\bzht{d}) \neq 0$ (i.e., if there are steps in $\mS$ having $z_d$ coordinate $0$) then the number of walks of length $n$ that never leave the non-negative orthant satisfies
\[ s_n =  S(\bone,\rho)^n \cdot n^{-d/2-1} \cdot C_{\rho} \left(1 + O(n^{-1})\right) . \]
If $Q(\bzht{d}) = 0$ then the number of walks of length $n$ that never leave the non-negative orthant satisfies
\[
s_n = n^{-d/2-1} \cdot {\Big[} S(\bone,\rho)^n \cdot C_{\rho} + S(\bone,-\rho)^n \cdot C_{-\rho} {\Big]}\left(1 + O(n^{-1})\right).
\]
\end{theorem}

Again, this result can be immediately applied to families of models. Note that $S(\bone,\rho) = Q(\bone) + 2 \sqrt{A(\bone)B(\bone)}$.

\begin{example}
Consider the step set where $A\left(\bzht{d}\right) = \prod_{j<d} (z_j + \oz_j)$, $Q\left(\bzht{d}\right)=0$, and $B\left(\bzht{d}\right)=1$ (note that this is the reflection in the $z_d$ axis of the step set of the previous example). Then $\rho = 2^\frac{d-1}{2}$ and
\[ C_{\rho} = \frac{2^{2d-3/2}}{\pi^{d/2}\left(2^{(d-1)/2}-1\right)^2} \qquad\qquad
C_{-\rho} = \frac{2^{2d-3/2}}{\pi^{d/2}\left(2^{(d-1)/2}+1\right)^2},
\]
so
\[ s_n = \left(2^{(d+1)/2}\right)^n \cdot n^{-d/2-1} \cdot \frac{2^{2d-3/2}}{\pi^{d/2}(2^{d-1}-1)^2} \cdot c_n \left(1 + O(n^{-1})\right),\]
where
\[ c_n = \begin{cases} 2^d+2 &: n \text{ is even} \\ 2^{(d+3)/2} &: n \text{ is odd} \end{cases} \]
\end{example}

Theorem~\ref{thm:NegAsm} is proved in Section~\ref{ss:negdrift}.  

\section{Lattice Path Generating Functions}
\label{sec:KernelMethod}

We now show how to derive several useful expressions for the generating functions of the lattice path models we consider. We closely follow the kernel method as outlined in Bousquet-M{\'e}lou and Mishna \cite{BoMi10}, and a straightforward generalization to higher dimensions discussed by Melczer and Mishna~\cite{ MelczerMishna2016}. This approach builds heavily on a probabilistic framework detailed by Fayolle, Iasnogorodski, and Malyshev~\cite{FaIaMa99}; see also Bousquet-M{\'e}lou and Petkov\v{s}ek~\cite{BoPe00} for a more general overview on the kernel method.

\subsection{A generating function expression via the kernel method}

To apply the kernel method we introduce the \emph{symmetry group} of the walk.

\begin{definition}
\label{def:mG}
For $1\leq j \leq d-1$ define the map $\sigma_j: \mathbb{C}^d \rightarrow \mathbb{C}^d$ by 
\[ \sigma_j\left(z_1,\dots,z_d \right) = (z_1,\dots,z_{j-1},\oz_j,z_{j+1},\dots, z_d), \]
and the map $\gamma:\mathbb{C}^d \rightarrow \mathbb{C}^d$ by 
\[ \gamma\left(z_1,\dots,z_d \right) = \left(z_1,\dots,z_{d-1}, \oz_d \frac{A\left(\bzht{d}\right)}{B\left(\bzht{d}\right)} \right). \]
We can view these maps as acting on Laurent polynomials $f \in \mathbb{C}[z_1,\oz_1,\dots,z_d,\oz_d]$ through
\[ \sigma \cdot f(\bz) := f(\sigma(z_1,\dots,z_d)) \]
and further view them as acting on elements $\sum_{n \geq 0} f_n(\bz)t^n \in \mathbb{C}[z_1,\oz_1,\dots,z_d,\oz_d][[t]]$ by
\[ \sigma \cdot \sum_{n \geq 0} f_n(\bz)t^n := \sum_{n \geq 0} \left(\sigma \cdot f_n(\bz)\right) t^n = \sum_{n \geq 0} f_n(\sigma(\bz))t^n.\]
Finally, we let $\mG$ be the group of birational transformations generated by $\sigma_1,\dots,\sigma_{d-1}$ and $\gamma$.
\end{definition}

\begin{remark}
Since $\mS$ is symmetric over all but one axis we have, for each $j=1,\dots,d-1$,
\[ \sigma_j\left(A\left(\bzht{d}\right)\right) = A\left(\bzht{d}\right) \qquad\qquad \sigma_j\left(B\left(\bzht{d}\right)\right) = B\left(\bzht{d}\right)\]
which, together with the fact that $\gamma$ fixes $S(\bz)$, implies that $S(\bz)$ is fixed by $\mG$.  Furthermore, these equalities show that the generators of $\mG$, which are involutions, commute, meaning $\mG$ is the finite group of order $2^d$ defined by
\[ \mG := \left\{ \sigma_1^{j_1} \cdots \sigma_{d-1}^{j_{d-1}} \gamma^{j_d} : j_1,\dots,j_d \in \{0,1\} \right\}. \]
The group $\mG$ is the direct sum of $d$ cyclic groups of order 2.
\end{remark}

Let $F(\bz,t)$ be the multivariate generating function 
\[ F(\bz,t) = \sum_{\substack{\bi \in \mathbb{N}^d \\ n \geq 0}} a_{\bi,n}\bz^i t^n, \]
where $a_{\bi,n}$ counts the number of weighted walks of length $n$ using the steps in $\mS$, beginning at the origin, ending at $\bi \in \mathbb{N}^d$, and never leaving the non-negative orthant in $\mathbb{Z}^d$.  Describing a walk of length $n$ ending at $\bi \in \mathbb{N}^d$ recursively as a walk of length $n-1$ followed by a single step, one can show (see Melczer and Mishna~\cite{ MelczerMishna2016}) that the generating function satisfies a functional equation of the form
\begin{equation} (1-tS(\bz))\bz^{\bo}F(\bz,t) = \bz^{\bo} + \sum_{k=1}^d L_k(\bzhat,t), \qquad L_k(\bzhat,t) \in \mathbb{Q}[\bzhat][[t]]. 
\label{eq:funform} \end{equation}
In particular, note that each $L_k(\bzhat,t)$ is independent of the variable $z_k$.

When manipulating the formal expressions that arise in our application of the kernel method, we may encounter rational functions in the variables $z_1,\dots,z_d$ which, in addition to not being analytic at the origin, are not Laurent polynomials in these variables.  Thus, we make use of the iterated Laurent series ring $\mR = \mathbb{Q}((z_1))\cdots((z_d))[[t]]$; unless otherwise stated all computations below are assumed to take place in the ring $\mR$, which contains both $\mathbb{Q}[z_1,\oz_1,\dots,z_d,\oz_d][[t]]$ and $\mathbb{Q}[[\bz,t]]$.  Note that every rational function in $\mathbb{Q}(\bz)$ has an expansion in $\mR$.  For further details on iterated Laurent series, including their uses in combinatorics and a classification of which formal series are iterated Laurent series, the reader is referred to the PhD thesis of Xin~\cite{Xin04}. We define the \emph{positive sub-series extraction} operator $[\bz^{\geq 0}] : \mR \rightarrow \mathbb{Q}[[\bz,t]]$ by
\[ [\bz^{\geq 0}] \sum_{n \geq 0} \left( \sum_{\bi \in \mathbb{Z}^d} a_{\bi,n}z^{\bi} \right)t^n := \sum_{n \geq 0} \left( \sum_{\bi \in \mathbb{N}^d} a_{\bi,n}z^{\bi} \right)t^n.\]
This setup leads to Theorem~\ref{thm:kernel}, typical of the kernel method (see Bousquet-M{\'e}lou and Mishna~\cite{BoMi10}, for instance, or Zeilberger~\cite{Zeilberger1983} and Gessel and Zeilberger~\cite{GeZe92} for similar expressions in a multivariate setting).

\begin{theorem} 
\label{thm:kernel}
If $\mS$ is symmetric over all but one axis, then the multivariate generating function $F(\bz,t)$ tracking endpoint and length satisfies
\begin{equation} 
F(\bz,t) = [\bz^{\geq 0}] \frac{\sum_{\sigma \in \mG} \sgn(\sigma)\sigma(z_1\dots z_d)}{(z_1 \cdots z_d)(1-tS(\bz))}, \label{eq:posext}
\end{equation}
where 
\[ \sgn\left(\sigma_1^{j_1} \cdots \sigma_{d-1}^{j_{d-1}} \gamma^{j_d}\right) = (-1)^{j_1 + \cdots + j_d}.\]
The generating function $F(\bz,t)$, and thus the specialized generating function $F(\bone,t)$ that counts the walks of a given length ending anywhere, are D-finite.
\end{theorem}

\textbf{Note:} The order of the iterated Laurent fields that define $\mR$ is important. If one works in an iterated Laurent field where $z_d$ is not the last variable before $t$, Equation~\eqref{eq:posext} may not hold.

\begin{proof}
We begin by examining the expression $\sigma(z_1\dots z_d)F(\sigma(\bz),t)$ for some fixed $\sigma = \sigma_1^{j_1} \cdots \sigma_{d-1}^{j_{d-1}} \gamma^{j_d} \in \mG$.  When $j_d = 1$, then every term in the expansion of $\sigma(z_1\dots z_d)F(\sigma(\bz),t)$ in the ring $\mR$ will have negative power of $z_d$ (due to the order of the variables used when defining $\mR$).  Otherwise, if $j_d = 0$ and there is some $k \in \{1,\dots,d-1\}$ such that $j_k =1$ then every term in the expansion of $\sigma(z_1\dots z_d)F(\sigma(\bz),t)$ in the ring $\mR$ will have negative power of $z_k$.
Thus, we see $[\bz^{\geq 0}]\sigma(z_1\dots z_d)F(\sigma(\bz),t) = 0$ for $\sigma \in \mG$ unless $\sigma$ is the identity element.  This implies
\begin{align*}
[\bz^{\geq 0}] \sum_{\sigma \in \mG} \sgn(\sigma)\sigma(z_1\dots z_d)F(\sigma(\bz),t) 
&= \sum_{\sigma \in \mG} \sgn(\sigma)\left([\bz^{\geq 0}]\sigma(z_1\dots z_d)F(\sigma(\bz),t)\right) \\
&= (z_1 \cdots z_d) F(\bz,t).
\end{align*}
By definition, for all $\sigma \in \mG$ and $\tau \in \{\sigma_1,\dots,\sigma_{d-1},\gamma\}$,
\[ \sgn(\tau\sigma) = -\sgn(\sigma). \]
As $S(\bz)$ is fixed by the elements of $\mG$, to prove Equation~\eqref{eq:posext} from Equation~\eqref{eq:funform} it is sufficient to show that for each $k=1,\dots,d$,
\[ \sum_{\sigma \in \mG} \sgn(\sigma) \sigma(z_1\dots z_d) \left( \sigma \cdot L_k(\bzhat,t) \right) = 0. \]
Fix $k$ and write $\mG$ as the disjoint union $\mG = \mG_0 \cup \mG_1$, where 
\begin{align*}
\mG_0 &= \left\{ \sigma_1^{j_1} \cdots \sigma_{d-1}^{j_{d-1}} \gamma^{j_d} : j_1,\dots,j_d \in \{0,1\}, j_k=0 \right\} \\
\mG_1 &= \left\{ \sigma_1^{j_1} \cdots \sigma_{d-1}^{j_{d-1}} \gamma^{j_d} : j_1,\dots,j_d \in \{0,1\}, j_k=1 \right\}.
\end{align*}
Then for all $g \in \mG_1$, $(\sigma_kg) \cdot L_k(\bzhat,t) = g \cdot L_k(\bzhat,t)$, and therefore
{\small
\begin{align*}
\sum_{\sigma \in \mG} \sgn(\sigma) \sigma(z_1\dots z_d)\left( \sigma \cdot L_k(\bzhat,t) \right) 
&= \sum_{\sigma \in \mG_0}\sgn(\sigma)  
\sigma(z_1\dots z_d)\left( \sigma \cdot L_k(\bzhat,t) \right) \\
&\quad + \sum_{\sigma \in \mG_1} \sgn(\sigma) \sigma(z_1\dots z_d)\left( \sigma \cdot L_k(\bzhat,t) \right) \\
&= \sum_{\sigma \in \mG_0} \left(\sgn(\sigma) - \sgn(\sigma) \right) \sigma(z_1\dots z_d)\left( \sigma \cdot L_k(\bzhat,t) \right) \\
&= 0,
\end{align*}
}
\noindent
as desired.  The results on D-finiteness follow from a classical result of Lipschitz~\cite{Li88} which states (in an equivalent form) that the class of D-finite functions is closed under positive sub-series extraction.
\end{proof}
The next result determines an explicit expression for the generating function under consideration.

\begin{lemma} \label{lem:orbitsum}
For the group $\mG$,
\[ \sum_{\sigma \in \mG} \sgn(\sigma)\sigma(z_1\dots z_d) = (z_1 - \oz_1) \cdots (z_{d-1} - \oz_{d-1}) \left(z_d - \oz_d \frac{A\left(\bzht{d}\right)}{B\left(\bzht{d}\right)}\right).\]
Consequently,
\[
F(\bz, t) = [\bz^{\geq 0}] R(\bz, t)
\]
where
\begin{equation}
\label{eq:posext2}
R(\bz, t) = \frac{(1-z_1^{-2}) \cdots (1-z_{d-1}^{-2}) \left(B\left(\bzht{d}\right)- z_d^{-2} A\left(\bzht{d}\right)\right)}{B\left(\bzht{d}\right)\left(1 - tS(\bz)\right)}.
\end{equation}
\end{lemma}
\begin{proof}
The first statement follows directly from the definition of $\mG$ and the sign operator (formally it can be proved by induction). The second statement comes from combining Lemma~\ref{lem:orbitsum} with \eqref{eq:posext}.
\end{proof}

\subsection{A diagonal representation}
\label{sec:diag}

Next, we turn back to the sequence counting the total number of walks of a given length (regardless of endpoint).  The generating function of this sequence is simply $F(\bone,t)$, since specializing each $z_j$ variable to 1 sums over its possible values.  

We may translate the positive sub-series extraction given by Equation~\eqref{eq:posext} into an expression for $F(\bone,t)$ using the diagonal operator $\Delta : \mR \rightarrow \mathbb{Q}[[t]]$ defined by 
\[ \Delta \left(\sum_{n \geq 0} \left( \sum_{\bi \in \mathbb{Z}^d} a_{\bi,n}z^{\bi} \right)t^n\right) := \sum_{n \geq 0} a_{n,\dots,n}t^n. \]
Our asymptotic results will follow from an analysis of a diagonal expression for $F(\bone, t)$. Establishing this diagonal expression is more complicated than in Melczer and Mishna~\cite{ MelczerMishna2016}, because we must consider expressions whose coefficients in $t$ are not Laurent polynomials. In the completely symmetric case $A = B$ in \eqref{eq:posext2}, and cancellation leaves only $1-tS(\bz)$ in the denominator.

The following technical lemma is elementary, involving only algebraic manipulations (see also Melczer and Mishna~\cite[Proposition 2.6]{ MelczerMishna2016}).

\begin{lemma} \label{lem:diagsubs}
Let $P(\bz,t) \in \mathbb{Q}[z_1,\oz_1,\dots,z_d,\oz_d][[t]] \subset \mR$.  Then 
\begin{equation} \label{eq:postodiag}
\left([\bz^\geq]P(\bz,t)\right) \bigg|_{z_1=1,\dots,z_d=1} = \Delta \left(\frac{P\left(\oz_1,\dots,\oz_d,z_1\cdots z_d\cdot t\right)}{(1-z_1)\cdots(1-z_d)}\right),
\end{equation}
where the diagonal on the right hand side is taken as an expansion in $\mR$, as usual.
\end{lemma} 

The proof follows from the definition of the diagonal after writing out the geometric series and expansion of $P$ on the right hand side.  

We would like to use Lemma~\ref{lem:diagsubs} directly, but $R$ in~\eqref{eq:posext2} does not lie in the correct ring and the substitutions indicated (replacing $z_i$ by $\oz_i$) are not formally justified. This problem can be circumvented through a tedious but elementary generating function argument, taking into account the precise structure of the rational function under consideration, yielding Proposition~\ref{prop:oldDiag}. Due to certain undesirable properties of the representation~\eqref{eq:oldDiag} we use a different diagonal expression for the generating function more suited to an asymptotic analysis, so we omit the proof.

\begin{proposition} \label{prop:oldDiag}
Let $\mS$ be a weighted step set satisfying the conditions above.  Then the generating function counting the number of walks of a given length in the lattice path model defined by $\mS$ satisfies 
\begin{equation} 
F(\bone, t) = 
\Delta \left(\frac{(1+z_1) \cdots (1+z_{d-1}) \left(B\left(\bzht{d}\right) - z_d^2A\left(\bzht{d}\right)\right)}{ (1-z_d) B\left(\bzht{d}\right) \left(1 - tz_1\cdots z_dS(z_1, \dots, z_{d-1}, \oz_d)\right)}\right).\label{eq:oldDiag}
\end{equation}
\end{proposition}

The rational function in~\eqref{eq:oldDiag} presents a challenge for the integral manipulations necessary to compute asymptotics as one can only easily deform domains of integration where the integrand is analytic; the factor $B\left(\bzht{d}\right)$ present in the denominator can give strange surfaces of singularities.  Instead we use the following alternative expression, which is a power series in $t$ with Laurent polynomial coefficients in the other variables. 

\begin{theorem} \label{thm:diag2}
Let $\mS$ be a weighted step set satisfying the conditions above.  Then the generating function counting the number of walks of a given length in the lattice path model defined by $\mS$ satisfies 
\[ F(\bone, t) =  \Delta \left( \frac{G(\bz,t)}{H(\bz,t)} \right),\] 
where
\begin{equation}
\label{eq:F2}
\begin{split}
G(\bz,t) &= (1+z_1) \cdots (1+z_{d-1}) \left(1 - t z_1\cdots z_d\left(Q\left(\bzht{d}\right) + 2z_d A\left(\bzht{d}\right)\right)\right)  \\
H(\bz,t) &= (1-z_d)  {\Big(}1 - tz_1\cdots z_d \oS(\bz){\Big)} {\Big(}1 - tz_1\cdots z_d \left(Q\left(\bzht{d}\right) + z_dA\left(\bzht{d}\right)\right){\Big)},
\end{split}
\end{equation}
and 
\[ \oS(\bz) = S(\bzht{d},\oz_d) = \oz_dB\left(\bzht{d}\right) + Q\left(\bzht{d}\right) + z_dA\left(\bzht{d}\right).\]
\end{theorem}

\begin{proof}
Expanding the expression in \eqref{eq:posext2} we obtain
\[
(1-\oz_1^2)\cdots (1 - \oz_{d-1}^2) \cdot \left( 1 - \oz_d^2 \frac{A\left(\bzht{d}\right)}{B\left(\bzht{d}\right)} \right) 
\cdot \sum_{n\geq 0} t^n \left(\oz_d A\left(\bzht{d}\right) + Q\left(\bzht{d}\right) + z_d B\left(\bzht{d}\right) \right)^n.
\]
As the sub-expression
\[
(1-\oz_1^2)\cdots (1 - \oz_{d-1}^2) \cdot \left( - \oz_d^2 \frac{A\left(\bzht{d}\right)}{B\left(\bzht{d}\right)} \right) \cdot 
\sum_{n\geq 0} t^n \left(\oz_d A\left(\bzht{d}\right) + Q\left(\bzht{d}\right)  \right)^n 
\]
contains no positive powers of $z_d$, we can subtract it from $R(\bz)$ and extract the positive part of
{\small 
\[
\frac{(1-\oz_1^2) \cdots (1-\oz_{d-1}^2) \left( 1- \oz_d^2 A\left(\bzht{d}\right)/B\left(\bzht{d}\right)\right)}{1 - tS(\bz)}
+ 
\frac{(1-\oz_1^2) \cdots (1-\oz_{d-1}^2) \left(\oz_d^2 A\left(\bzht{d}\right)/B\left(\bzht{d}\right)\right)}{1 - t\left(\oz_d A\left(\bzht{d}\right) + Q\left(\bzht{d}\right) \right)}.
\]
}
This final rational function simplifies to
\[
\frac{(1-\oz_1^2) \cdots (1-\oz_{d-1}^2)\left(1 - t\left(2\oz_d A\left(\bzht{d}\right) + Q\left(\bzht{d}\right)\right) \right)}{\left( 1 - t \left(\oz_d A\left(\bzht{d}\right) + Q\left(\bzht{d}\right) + z_d B\left(\bzht{d}\right) \right) \right)\left( 1 - t \left(\oz_d A\left(\bzht{d}\right) + Q\left(\bzht{d}\right) \right) \right)},
\]
and we can now apply Lemma~\ref{lem:diagsubs}. 
\end{proof}

Note that the power series expansion of $1/H(\bz,t)$ has all non-negative coefficients, which will allow us to simplify necessary characterizations of the singularities of $G(\bz,t)/H(\bz,t)$ below. In the special case where $S$ is symmetric over all axes, we obtain an expression different from that in \cite{MelczerMishna2016}; by forcing positivity on our series coefficients we have lost some symmetry and less cancellation occurs. For example, the generating function of the model with all possible steps in 2 dimensions is the diagonal of
\[
\frac{(1+x)(1+y)}{1-t(1+x+y+x^2+y^2+x^2y+xy^2+x^2y^2)}
\]
using the expression in Proposition~\ref{prop:oldDiag}, which coincides with that in \cite{MelczerMishna2016}, but the diagonal of 
{\small
\[
\frac{(1+x)(1-2t(y+y^2+x^2y+xy^2+x^2y^2))}{(1-y)(1-t(1+x+y+x^2+y^2+x^2y+xy^2+x^2y^2))(1-t(y+y^2+x^2y+xy^2+x^2y^2))}
\]
}
using the expression in Theorem~\ref{thm:diag2}.

\subsection{Models whose step sets have fewer symmetries}
\label{ss:why_not}

We end this section with a justification of why we only consider models missing one symmetry (instead of two, three, etc.).  Indeed, as the following theorem shows, in any dimension $d \geq 2$ there exists a model that is missing two symmetries and admits a generating function that is not D-finite.  As the diagonal of a multivariate rational function must be D-finite~\cite{Ch88,Li88}, this shows that it is impossible to determine the asymptotics of all models missing two symmetries uniformly through multivariate rational diagonals and analytic combinatorics in several variables.

\begin{theorem} There exists a sequence of step sets $\mS_2,\mS_3,\dots$ with $\mS_d$ defining a step set of dimension $d$ that is symmetric over all but two axes, such that the generating function of each model is non D-finite.
\end{theorem} 

\begin{proof}
If a counting sequence $(c_n)_{n\geq0}$ has asymptotics of the form $c_n \sim K \cdot \rho^n \cdot n^\alpha$ for constants $K,\rho,\alpha \in \mathbb{R}$ and its generating function is $D$-finite, then $\rho$ is algebraic and $\alpha$ is rational (see Theorem 3 of Bostan, Raschel, and Salvy~\cite{ BoRaSa14}). 

When $d=2$, consider the set of steps 
\[ \mS_2 = \{(-1,-1),(0,-1),(0,1),(1,0),(-1,0)\}.\]  
Bostan, Raschel, and Salvy~\cite{ BoRaSa14} show that $(e_n)_{n \geq 0}$, the number of walks on these steps staying in the first quadrant that begin and end at the origin, has dominant asymptotics
\[ e_n \sim K_e \cdot \rho_e^n \cdot n^{\alpha_e} \]
where $\alpha$ is an irrational number (approximately 2.757466) equal to $-1-\pi/\arccos(-c)$, with $c$ an algebraic number satisfying $c^3-c^2+(3/4)c-(1/8)=0$.  Work of Duraj~\cite{Du14} implies in our context\footnote{That article takes a probabilistic view of exit times for random walks to leave certain cones, and applies to a wide range of models; see its Example 7 for the case of two dimensional random walks in a quadrant.} that---since $\mS_2$ has negative vector sum in both coordinates---the sequence counting the total number of steps has dominant asymptotics 
\[ s^{(2)}_n \sim K_2 \cdot \rho_2^n \cdot n^{\alpha_2}, \]
where $\alpha_2 = \alpha_e$ and is thus irrational.

For $d \geq 3$ let $\mS_d = \mS_2 \times \{\pm1\}^{d-2}$.  Every walk of length $n$ on the steps $\mS_d$ is constructed uniquely from a walk of length $n$ on the steps $\mS_2$ in the non-negative quadrant and $d-2$ independent walks of length $n$ on the steps $\{-1,1\}$ restricted to the non-negative integers (this is a simple version of the Hadamard decomposition of walks studied in Bostan et al.~\cite{BostanBousquet-MelouKauersMelczer2016}).  Thus, the number of walks of length $n$ taking steps in $\mS_d$ restricted to the $d$-dimensional non-negative orthant is 
\[ s^{(d)}_n = s^{(2)}_n \cdot c_n^{d-2}, \]
where $c_n$ is the number of Dyck paths that do not have to end at 0 (sometimes called Dyck prefixes).  It is a classical result of enumerative combinatorics that $c_n = \binom{n}{\lceil n/2 \rceil}$ with dominant asymptotics of the form
\[ c_n \sim \sqrt{2/\pi}\cdot2^n\cdot n^{-1/2}, \]
which implies 
\[ s^{(d)}_n \sim K_d \cdot (\rho_d)^n \cdot n^{\alpha_d}, \]
with $\alpha_d = \alpha_2 - d/2+1 \not\in\mathbb{Q}$.
\end{proof}
It would be of great interest to find `simple' diagonal expressions involving more general multivariate meromorphic functions for walk models with non-D-finite generating functions. Such multivariate functions could not be D-finite, in the sense that the vector space of all partial derivatives over $\mathbb{Q}(\bz)$ would need to be infinite dimensional.

Furthermore, although not all models missing two symmetries can be handled directly by our methods, there are some models missing two (or more) symmetries whose generating functions can be written as rational diagonals. For a specific model, one can attempt to follow the algebraic kernel method for higher dimensional walks; see, for instance, Bostan et al.~\cite{BostanBousquet-MelouMelczer2018} for a general framework. Characterizing all such models is a more difficult task. Our best guess is that this property is related to being a Hadamard decomposition, in the sense of Bostan et al.~\cite{BostanBousquet-MelouKauersMelczer2016},  of some (hopefully nice) characterizable family of `atomic' D-finite models. Although in general one cannot simply determine asymptotics of a model which admits a Hadamard decomposition by multiplying the asymptotics of lower dimensional sequences, many properties such as D-finiteness are inherited via Hadamard decomposition; see Bostan et al.~\cite[Section 5]{BostanBousquet-MelouKauersMelczer2016} for details. Given a model whose generating function can be written as a rational diagonal, Courtiel et al.~\cite{CourtielMelczerMishnaRaschel2017} develop a method to determine weightings of the step set so that the weighted generating function can be represented as a parametrized rational diagonal with the weights as parameters.

We now move on to the analysis of the expressions obtained using the methods of analytic combinatorics in several variables. We use the methods developed by Pemantle and Wilson~\cite{PeWi13} for asymptotics controlled by points where the zero set of $H(\bz,t)$ is locally a manifold or a union of finitely many transversely intersecting manifolds.

\section{Contributing Singularities}
\label{sec:contrib}

Suppose $Q(\bz,t)$ is a rational function analytic at the origin. As in the univariate case, a multivariate singularity analysis starts from the Cauchy integral formula, which implies
\begin{equation}
b_n := [(z_1\cdots z_d\, t)^n]Q(\bz,t) = \frac{1}{(2\pi i)^{d+1}}\int_{\mC} Q(\bz) \cdot \frac{d\bz\,dt}{(z_1\cdots z_d \, t)^{n+1}}
\label{eq:mCIF}
\end{equation}
for any $n \in \mathbb{N}$ and $\mC$ a product of circles sufficiently close to the origin.  If $\mD$ is the domain of convergence of the power series of $Q(\bz,t)$ at the origin, and $\mV$ is the set of singularities of $Q(\bz,t)$, then any singularity on the boundary $\partial \mD$ of $\mD$ is called \emph{minimal}.  When $P(\bz,t)$ is a polynomial we say that $(\bw,s)$ is a \emph{minimal zero} of $P$ if $P(\bw,s) = 0$ and $P(\by,r) \neq 0$ whenever
\[ |w_j| \leq |y_j|, \quad j=1,\dots,d, \qquad |r| \leq |s| \]
and one of the inequalities is strict.  Note that a minimal point of $Q$ is a minimal zero of its denominator, and vice-versa.

As $Q(\bz,t)$ is rational, $b_n$ grows at most exponentially and standard integral bounds imply
\begin{equation} \limsup_{n \rightarrow\infty} |b_n|^{1/n} \leq |w_1 \cdots w_d \, s|^{-1} \label{eq:exbound} \end{equation}
for any minimal point $(\bw,s) \in \partial \mD \cap \mV$.  In the simplest cases, one hopes to identify a finite set of minimal points achieving the optimal bound in Equation~\eqref{eq:exbound}.  When such a set exists, and the local geometry of the algebraic set $\mV$ is sufficiently nice, asymptotics can then be determined.

We now specialize our arguments to the rational function $Q(\bz,t) = G(\bz,t)/H(\bz,t)$ defined by Theorem~\ref{thm:diag2}; note that $G$ and $H$ are co-prime, so the singularities of $Q$ are the zeros $\mV = \mV(H)$ of the polynomial $H$.  Because of the nice form of $H$, we are able to characterize its minimal zeros achieving the best bound in Equation~\eqref{eq:exbound}, which is typically the hardest step of any asymptotic analysis. We make use of the following result.

\begin{lemma}
\label{lem:minimum}
Suppose $\mathcal{P} \subset \mathbb{Z}^d$ is a finite set not contained in a hyperplane of $\mathbb{R}^d$, and $a_{\bi}>0$ are positive constants for each $\bi \in \mathcal{P}$.  Then every critical point of
\[ P(\bz) = \sum_{\bi \in \mathcal{P}}a_{\bi}\bz^{\bi} \]
on $\left(\mathbb{R}_{>0}\right)^d$ is a global minimum and $P$ admits at most one critical point on this domain.  Furthermore, such a global minimum exists if and only if $\mathcal{P}$ is not contained in a halfspace containing the origin. 
\end{lemma}
\begin{proof}
This result follows from the strict convexity of the \emph{Laplace transform} $P\left(e^{x_1},\dots,e^{x_d}\right)$; for details see Garbit and Raschel~\cite[Lemma 7]{GarbitRaschel2016}. 
\end{proof}

In order to reason about minimal zeros of $H(\bz,t)$, we define the factors 
\begin{align*}
H_1 & := 1 - tz_1\cdots z_d\oS(\bz) = 1 - tz_1\cdots z_{d-1} \left( z_d^2A(\bzht{d})+ z_d Q(\bzht{d}) + B(\bzht{d})\right) \\
H_2 & := 1 - tz_1\cdots z_d \left(Q(\bzht{d}) + z_dA(\bzht{d}) \right) \\
H_3 & := 1 - z_d,
\end{align*}
and set $\mV_j = \mV(H_j)$. 

Under our assumptions on $\mS$ the conditions of Lemma~\ref{lem:minimum} are satisfied by $\oS(\bz)$, giving the following.

\begin{proposition}
\label{prop:minpts}
The unique minimal zero of $H(\bz,t)$ with positive coordinates that minimizes $|z_1 \cdots z_d \, t|^{-1}$ is
\begin{align*}
\bp_1&:=\left(1, 1, \dots, 1, \sqrt{\frac{B(\bone)}{A(\bone)}}, \frac{\sqrt{A(\bone)/B(\bone)}}{2\sqrt{A(\bone)B(\bone)}+Q(\bone)}\right) \qquad \text{if the drift is negative} \\[+2mm]
\bp_2&:= \left(1, 1, \dots, 1, \frac{1}{S(\bone)} \right) \qquad \text{otherwise}.
\end{align*}
\end{proposition}

\begin{proof}
As $|z_1 \cdots z_d \, t|^{-1}$ decreases as $(\bz,t)$ moves away from the origin, any such minimizer must be a zero of $H_1$ or $H_2$.  Since $\oS(\bz)$ has non-negative coefficients, any zero of $H_1$ with positive coordinates is a minimal zero as $t = (z_1 \cdots z_d\oS(\bz))^{-1}$ increases as one of the $z_j$ decreases and the others are constant.  Furthermore, on $\mV_1 \cap \left(\mathbb{R}_{>0}\right)^d$
\[ |z_1 \cdots z_d \, t|^{-1} = \oS(\bz), \]
which by Lemma~\ref{lem:minimum} has a unique minimum corresponding to a unique critical point. The system
\[ \oS_{z_1}(\bz) = \cdots = \oS_{z_d}(\bz) = 0 \]
can be reduced to
\[ (1-z_1^2) \cdot [z_1^{-1}]\oS(\bz) \,=\, \cdots = (1-z_{d-1}^2) \cdot [z_{d-1}^{-1}]\oS(\bz) \,=\, B(\bzht{d}) - z_d^2A(\bzht{d}) \,=\, 0,\]
which has only the solution $(\bone,\sqrt{B(\bone)/A(\bone)})$ with positive coordinates, as $\oS$ has all non-negative coefficients.

If the drift is non-positive, $B(\bone) \leq A(\bone)$ and $\bp_1$ is a minimal zero of the product $H_1(\bz,t)H_3(\bz,t)$. Otherwise, any minimal zero of $H_1(\bz,t)H_3(\bz,t)$ that minimizes $|z_1 \cdots z_d \, t|^{-1}$ must lie on $\mV_1 \cap \mV_3$, where 
\[ |z_1 \cdots z_d \, t|^{-1} = \oS(\bzht{d},1), \]
and Lemma~\ref{lem:minimum} implies the minimizer is $\bp_2$. 

Finally, if $(\bz,t) \in \mV_2 \cap \left(\mathbb{R}_{>0}\right)^d$ then
\[ t = \frac{1}{z_1 \cdots z_d \left(Q(\bzht{d}) + z_dA(\bzht{d}) \right)} > \frac{1}{z_1 \cdots z_d\oS(\bz,t)} \]
since $z_1 \cdots z_dB(\bz) > 0$.  But this implies $(\bz,t)$ is not a minimal zero of $H$, as there exists $(\bz,s) \in \mV_1$ with $0<s<t$.
\end{proof}

In order to perform a local singularity analysis we will need to describe $\mV$ near points of interest. In our case, the singular set $\mV$ is the union of smooth manifolds $\mV_1,\mV_2,$ and $\mV_3$ (for each $i$, the gradient of $H_i$ never vanishes when $H_i=0$). Furthermore, we show in the proof of Theorem~\ref{thm:contrib} that any minimal singularity will not lie on $\mV_2$, so that any minimal singularity is either in $\mV_1$ alone, $\mV_3$ alone, or the intersection $\mV_1 \cap \mV_3$. As the gradients of $H_1$ and $H_3$ are linearly independent at any common zero, we say $\mV_1$ and $\mV_3$ are \emph{transverse}.

In this setting, the \emph{stratum} of minimal $\bw \in \mV$ is the intersection of the $\mV_j$ containing $\bw$. Minimal $\bw \in \mV$ with non-zero coordinates is called a \emph{minimal critical point} if it is a critical point (in the differential geometry sense) of the map $\phi(\bz) = \log(z_1 \cdots z_d)$ from the stratum of $\bw$ to $\mathbb{C}$. Algebraically, this means that the gradient of $\phi(\bz) = \log(z_1 \cdots z_d)$ at $\bz=\bw$ can be written as a linear combination of the gradients of the $H_j$ polynomials defining the strata of $\bw$. Critical points are those where the Cauchy integral can be locally manipulated into a so-called Fourier-Laplace integral, where saddle-point methods can be applied to obtain asymptotics.

General definitions of critical and contributing points, where local function behaviour dictates coefficient asymptotics, can be found in~\cite{PeWi13}. In particular, Proposition 10.3.6 of~\cite{PeWi13} gives an explicit characterization of contributing points: in our setting, a singularity of $Q(\bz,t)$ is a \emph{contributing point} if it is a minimal critical point that minimizes $|z_1 \cdots z_d \, t|^{-1}$ on $\partial \mD$ (the exponential order of the asymptotic contribution of that point is maximum).

\begin{theorem}
\label{thm:contrib}
When the drift is positive, there are at most $2^{d-1}$ contributing points.  The point $\bp_2$ is one, and the others are the points $(\bw,1,t)$ where
\[ \bw \in \{\pm1\}^{d-1}, \quad t = \frac{1}{w_1\cdots w_{d-1} \cdot S\left(\bw,1\right)}, \text{ and} \quad |t| = \frac{1}{S(\bone,1)}. \]

When the drift is negative, there are at most $2^{d+1}$ contributing points.  The point $\bp_1$ is one, and the others are the points $(\bw,w_d,t)$ where
\[ \bw \in \{\pm1\}^{d-1}, \qquad w_d = \nu \sqrt{\frac{B(\bw)}{A(\bw)}}, \qquad t = \frac{1}{w_d S\left(\bw,\overline{w}_d\right)}, \]
\[ |w_d| = \frac{\sqrt{B(\bone)}}{\sqrt{A(\bone)}}, \quad \text{ and } \quad |t| = \frac{\sqrt{A(\bone)}}{\sqrt{B(\bone)}S\left(\bone,\sqrt{A(\bone)/B(\bone)}\right)},\]
with $\nu$ a fourth root of unity (note that in order to satisfy the condition on $|t|$ it is necessary that $B(\bw)/A(\bw)>0$, so the square root can be taken unambiguously).

When the drift is zero, there are at most $2^d$ contributing points.  The point $\bp_1=\bp_2$ is one, and the others are the points $(\bw,w_d,t)$ where
\[ (\bw,w_d) \in \{\pm1\}^d, \qquad t = \frac{1}{w_1 \cdots w_d S\left(\bw,w_d\right)}, \text{ and} \quad |t| = \frac{1}{S(\bone,1)}. \]
\end{theorem}

\begin{proof}
As the power series expansion of $1/H(\bz,t)$ has non-negative coefficients, every minimal point has the same coordinate-wise modulus as an element of $\mV$ with positive coordinates (an element of $\partial \mD$ is the limit of a sequence in $\mD$ that makes the power series of $1/H$ approach infinity, but as the power series coefficients are non-negative the series only gets larger when each coordinate is replaced by its modulus). 

Thus, we search for points in $\mV$ with the same coordinate-wise modulus as $\bp_1$ and $\bp_2$. First, we note that no point in $\mV_2$ is minimal as its $t$ variable will be smaller than required.  On $\mV_1$, we seek points $(\bw,s)$ such that 
\[ |\oS(\bw,s)| = \oS(\bp_1) \qquad \text{or} \qquad |\oS(\bw,s)| = \oS(\bp_2). \]
Since $z_1 \cdots z_d\oS(\bz,t)$ is a polynomial with non-negative coefficients, the triangle inequality implies the only way this can happen is if every monomial of $\oS(\bz,t)$ has the same argument when evaluated at $\bw$. Our assumptions on $\mS$ imply that $w_1,\dots,w_{d-1}$ must be real (and thus $\pm1$) so the points in the statement of Theorem~\ref{thm:contrib} are the only potential minimizers of $|z_1 \cdots z_d \, t|^{-1}$, and are minimal points.

Computing the gradient of $H_1(\bz) = 1 - tz_1\cdots z_d\oS(\bz)$ shows that a point $(\bz,t)$ with stratum $\mV_1$ is critical if and only if $\bz$ satisfies 
\[ \oS_{z_1}(\bz) = \cdots = \oS_{z_d}(\bz) = 0, \]
while a point with stratum $\mV_1 \cap \mV_3$ is critical if and only if
\[ \oS_{z_1}(\bz) = \cdots = \oS_{z_{d-1}}(\bz) = 0, \quad z_d=1. \]
These equations are satisfied by the stated points, therefore we have found the set of contributing points.
\end{proof}

\section{Asymptotic Expansions}
\label{sec:Asym}
The results of Pemantle and Wilson~\cite{PeWi13} apply broadly to compute asymptotics when contributing points are known.  Now that the set of contributing points is characterized by Theorem~\ref{thm:contrib} it is a straightforward (though computationally intensive) matter to compute asymptotics, which we do for each of the cases in Theorem~\ref{thm:contrib}. Some of the technical and laborious proofs that are not of theoretical interest are given in Appendix~\ref{appendixA}. 

We make an exponential change of variables to convert complex contour integrals to integrals over $\mathbb{R}^d$. To this end we introduce some basic notation.
\begin{definition}
\label{def:E}
For $\bp\in \mathbb{C}^d$, define $E$ on $[-\pi, \pi]^d$ by
$$
E_\bp(\bt)  = (p_1e^{i\theta_1}, \dots, p_de^{i\theta_d})\\
$$
For fixed $\bp$, every function $f(\bz)$ of $\bz$ yields a corresponding function $\tilde{f}(\bt):=f\circ E_\bp$ of $\bt$ under this change of variable.
\end{definition}

For $1\leq j \leq d$, we use the usual notation $\partial_j$ for the partial differential operator $(\partial/\partial\theta_j)$ and define the functions $B_1(\bzht{1}),\dots,B_{d-1}(\bz_{\,\widehat{d-1}})$ by stipulating that $B_k$ is the unique Laurent polynomial such that
\begin{equation} \oS(\bz) = (z_k+\oz_k)B_k(\bzht{k}) + Q_k(\bzht{k}) \label{eqn:Bj} \end{equation}
for some Laurent polynomial $Q_k$. For notational convenience we set $B_d(\bzht{d}) := B(\bzht{d})$.

\begin{remark}
We note that for each $j \neq k$ with $j,k<d,$ 
each $B_j$ is symmetric in $z_k$ and $\oz_k$, and does not involve any power of $z_j$. Also $A, B, Q$ (and both $S$ and $\oS$) are symmetric in $z_k$ and $\oz_k$. 
\end{remark}

We also need the following quantities.
\begin{definition}
\label{def:Bpj}
For each $j<d$, define Laurent polynomials $A'_j, B'_j, A''_j, B''_j$ by 
\begin{align*}
A(\bzht{d}) &= (z_j+\oz_j)A'_j(\bzht{j}) + A''_j(\bzht{j})\\
B(\bzht{d}) &= (z_j+\oz_j)B'_j(\bzht{j}) + B''_j(\bzht{j}).
\end{align*}
\end{definition} 

Finally, for a differentiable function $f(\bz)$ we define
\[ \gradlog f(\bz) := \left( z_1 \partial_1 f, \dots, z_d \partial_d f \right). \]
We now have all the necessary tools to compute asymptotics, beginning with the positive drift case. 

\subsection{The Positive Drift Models}
\label{ss:posdrift}

In the positive drift case, when $A(\bone) < B(\bone)$, Theorem~\ref{thm:contrib} implies that we are dealing with contributing points on the stratum $\mV_1 \cap \mV_3$. The next result follows from Theorem 10.3.4 of Pemantle and Wilson~\cite{PeWi13}, where $\mathbf{e}_j$ is the $j$th standard basis vector (with a 1 in its $j$th entry and zeros elsewhere).

\begin{proposition} \label{prop:PosAsm}
Let $\Gamma$ be the square matrix whose first $2$ rows are $\gradlog H_1(\bp)$ and $\gradlog H_3(\bp)$, and whose last $d-1$ rows are $p_j \mathbf{e}_j$ for $j=1,\dots,d-1$.  Furthermore, define
\[ g(\bt) := \log\left( \frac{1}{(p_1\cdots p_d)e^{i(\theta_1 + \cdots +\theta_d)}\oS(p_1e^{i\theta_1},\dots,p_de^{i\theta_d})} \right). \]
Then 
{\small
\[ [t^{n\bone}]\Delta Q(\bz,t) = \bp^{n \bone} \cdot n^{-(d-1)/2} \cdot \left( (2\pi)^{-(d-1)/2}(d+1)^{-(d-1)/2} \frac{G(\bp)}{\det \Gamma \cdot \sqrt{\det g''(\bzer)}} +  O(n^{-1}) \right), \]
}
where $g''(\bzer)$ denotes the Hessian of $g(\bt)$ at the origin.
\end{proposition}

\begin{proof}
As $\bp_2$ is the only contributing point where $G(\bz,t)$ does not vanish, Theorem 10.3.4 of Pemantle and Wilson~\cite{PeWi13} gives the above asymptotic result and the bound of $O(n^{-1})$ on the lower order terms.  
\end{proof}

Because $G(\bz,t)$ vanishes at all contributing points except for $\bp_2$, no positive drift model will have periodicity in its leading asymptotic term. Applying Proposition~\ref{prop:PosAsm} in our situation gives asymptotics in the positive drift case.

\begin{proof}[Proof of Theorem~\ref{thm:PosAsm}]
Theorem~\ref{thm:contrib} implies that the point $\bp_2 = (\bone,1,1/S(\bone,1))$ is the unique contributing point at which $G(\bz,t)$ does not vanish.  At this point, one can calculate that 
\[ \Gamma = \begin{pmatrix} 0 & 0 & 0 & \cdots & 0 & -1 & 0 \\ -1 & -1 & -1 & \cdots & -1 & -r & -1 \\ 1 & 0 & 0 & \cdots & 0 & 0 & 0 \\ 0 & 1 & 0 & \cdots & 0 & 0 & 0\\ 0 & 0 & 1 & 0 & \cdots & 0 & 0 \\ \vdots & \ddots & \ddots & \ddots & \ddots & \vdots & \vdots \\ 0 & 0 & 0 & \cdots & 1 & 0 & 0 \end{pmatrix}\]
for a real number $r<1$, which does not appear in the determinant of $\Gamma$, and 
\[\frac{G(\bone,1/|S|)}{H_2(\bone,1/|S|)} = 2^{d-1}\frac{\left(B(\bone) - A(\bone)\right)}{S(\bone)} \frac{S(\bone)}{B(\bone)} = 2^{d-1} \left(1 - \frac{A(\bone)}{B(\bone)}\right). \]
The Chain Rule and Lemma~\ref{lem:partialS} of Appendix~\ref{appendixA} imply that $g''(\bzer)$ is the diagonal matrix
\[ g''(\bzer) = \begin{pmatrix} 
\frac{2 B_1(\bone)}{(d+1)S(\bone)} & 0 & 0 & \cdots & 0 \\
0 & \frac{2 B_2(\bone)}{(d+1)S(\bone)} & 0 & \cdots & 0 \\
\vdots & \ddots & \ddots & \ddots & \vdots \\
0 & 0 & \cdots & \frac{2 B_{d-2}(\bone)}{(d+1)S(\bone)} & 0 \\
0 & 0 & \cdots & 0 & \frac{2 B_{d-1}(\bone)}{(d+1)S(\bone)} 
\end{pmatrix},
 \]
so that Proposition~\ref{prop:PosAsm} gives
\begin{equation*}
\begin{split} s_n & = S(\bone)^n  n^{-1/2} (2\pi)^{-\frac{d-1}{2}} (d+1)^{\frac{-(d-1)}{2}} \cdot \\
 & \frac{2^{d-1} \left(1 - \frac{A(\bone)}{B(\bone)}\right)}{\sqrt{(d+1)^{-(d-1)}2^{d-1}S(\bone)^{-(d-1)/2}(b_1\cdots b_{d-1})}} +  O(S(\bone)^n n^{-3/2}),
\end{split}
\end{equation*}
which simplifies to~\eqref{eqn:PosAsm}.
\end{proof}

\subsection{The Negative Drift Models}
\label{ss:negdrift}

In the negative drift case, when $A(\bone) < B(\bone)$, Theorem~\ref{thm:contrib} implies that we are dealing with contributing points on the stratum $\mV_1$ where $\mV$ itself is locally a manifold.  This simplifies computations, but an added difficulty is that the numerator vanishes to at least first order at every critical point.  We now state a general theorem that allows one to calculate asymptotics under these conditions, coming from Raichev and Wilson~\cite{RaWi08} (note that for us the dimension $d$ is one less than the number of variables of $F$).

\begin{theorem}
\label{thm:multasm}
Fix natural number $N>0$ and recall the above notation from this section. In a neighborhood in $\mV$ of a smooth critical point $\bp$ on $\mV_1$, write $t = h(\bz)$. Define  $u, \tilde{g}$ and $\underline{\tilde{g}}$ by
\begin{align*}
u(\bz) & = - \frac{G(\bz, h(\bz))}{h(\bz) (\partial H/\partial t)(\bz, h(\bz))} \\
\tilde{g}(\bt) &= \log \frac{\tilde{h}(\bt)}{\tilde{h}(\bzer)} + i \sum_{j=1}^d \theta_j  \\
\underline{\tilde{g}}(\bt) &= \tilde{g}(\bt) - \frac{1}{2} \bt^T \tilde{g}''(\bzer)\bt.
\end{align*}
Supposing that the Hessian determinant $ \det \tilde{g}''(\bzer)\neq \bzer$, define
\begin{equation}
\Psi^{(\bp)}_{n,N} := \bp^{-n \bone}\cdot n^{-d/2}\cdot (2\pi)^{-d/2}(\det \tilde{g}''(\bzer))^{-1/2}
\sum_{k=0}^{N-1} n^{-k} L_k(\tilde{u},\tilde{g})  , \label{eqn:Psi}
\end{equation}
where 
\[ L_k(\tilde{u},\tilde{g}) = \sum_{l=0}^{2k} \frac{\mathcal{H}^{k+l}(\tilde{u}\underline{\tilde{g}}^l)(\bzer)}{(-1)^k2^{k+l}l!(k+l)!} \]
with $\mathcal{H}$ the differential operator
\[ \mathcal{H} = - \sum_{1 \leq a,b \leq d} (\tilde{g}''(\bzer)^{-1})_{a,b}\partial_a\partial_b. \] Then, as $n\rightarrow\infty$,
\[ [t^{n\bone}]\Delta Q(\bz,t) = \sum_{\bp \in W} \Psi^{(\bp)}_{n,N}  + O(n^{-N}).\]
\end{theorem} 

Lemma~\ref{lem:horm_simp} of Appendix~\ref{appendixA} shows that the term $L_1(\tilde{u},\tilde{g})$, which will determine dominant asymptotics for negative drift models, simplifies considerably for the functions we consider.  In the setting of this section we have
\begin{align*}
\tilde{u}(\bt) &= \frac{(1+p_1e^{i\theta_1}) \cdots (1+p_{d-1}e^{i\theta_{d-1}})\left(1 - p_d^2e^{2i\theta_d}A(\bpht{d})/B(\bpht{d})\right)}{1 - p_de^{i\theta_d}} \\
\tilde{g}(\bt) &= \log\oS(\bp) - \log\oS(p_1e^{i\theta_1},\dots,p_de^{i\theta_d}),
\end{align*}
and Lemma~\ref{lem:partialS} of Appendix~\ref{appendixA} implies that $\tilde{g}''(\bzer)$ is the diagonal matrix
\[ \tilde{g}''(\bzer) = \begin{pmatrix} 
\frac{2 p_1 B_1(\bp_{\hat{1}})}{\oS(\bp)} & 0 & 0 & \cdots & 0 \\
0 & \frac{2 p_2 B_2(\bp_{\hat{2}})}{\oS(\bp)} & 0 & \cdots & 0 \\
\vdots & \ddots & \ddots & \ddots & \vdots \\
0 & 0 & \cdots & \frac{2 p_{d-1} B_{d-1}(\bp_{\widehat{d-1}})}{\oS(\bp)} & 0 \\
0 & 0 & \cdots & 0 & \frac{2 B_d(\bp_{\hat{d}})}{ p_d \oS(\bp)}
\end{pmatrix}.
 \]

Extensive calculations, given by Proposition~\ref{prop:NegAsm} and~Proposition~\ref{prop:NegAsm2} in Appendix~\ref{appendixA}, then give the following.

\begin{proposition}
Let $\mS$ be a step set that is symmetric over all but one axis and takes a step forwards and backwards in each coordinate, and let $W$ be the set of contributing points given by Theorem~\ref{thm:contrib}.  If $\mS$ has negative drift, then the number of walks of length $n$ that never leave the non-negative orthant satisfies
\begin{equation} s_n = \sum_{\bp \in W} 
\left[ (p_1\cdots p_dp_t)^{-n}  n^{-d/2-1} \left(K_\bp C_\bp + O(n^{-1})\right) \right],
\end{equation}
where
\begin{align*}
K_\bp & = 2^{-d} \, \pi^{-d/2} \, \oS(\bp)^{d/2} \, \left(p_1 \cdots p_{d-1} \cdot B_1(\bp_{\hat{1}}) \cdots B_{d-1}(\bp_{\widehat{d-1}}) \, B(\bp_{\widehat{d}})/p_d\right)^{-1/2} \\[+2mm]
C_{\bp} &= \frac{\overline{S}(\bp)\prod_{j<d} (1+p_j)}{1-p_d} \left[ \frac{1}{A(\bp)p_d(1-p_d)} + \sum_{j=1}^{d-1} \frac{1-p_j}{2p_jB_j(\bp)}\left(\frac{A'_j(\bp)}{A(\bp)} - \frac{B'_j(\bp)}{B(\bp)} \right)\right].
\end{align*}
\end{proposition}

Examining the set of contributing points given by Theorem~\ref{thm:contrib} implies that only those whose first $d-1$ coordinates are 1 contribute to dominant asymptotics (otherwise they have a $-1$ coordinate and $C_{\bp}$ is zero). Furthermore, if $c=\sqrt{B(\bone)/A(\bone)}$ and $\left|S(\bone,c\,\nu)\right| = S(\bone,c)$ then $\nu$ must be 1 if $Q \neq 0$ and must be either 1 or $-1$ if $Q=0$. Putting everything together gives Theorem~\ref{thm:NegAsm}. Note that $\rho$ is the reciprocal of $c$, so $\oS$ is replaced by $S$ in the theorem statement, and the radicand appearing in $C_{-\rho}$ is positive as $S(\bone,-\rho)$ and each $b_j(\bone,-\rho)$ are negative when $Q=0$. Applying Proposition~\ref{prop:PosAsm} in our situation gives asymptotics in the positive drift case. Because two contributing singularities may dictate the dominant asymptotic term, there are negative drift models whose leading asymptotic constant depends on the parity of $n$.

\section{Applications to 2D Models Restricted to a Quadrant}
\label{sec:2DWalks}

The study of two dimensional lattice walks restricted to the non-negative quadrant has been a very active area of interest in several sub-areas of combinatorics (see, for instance, the citations in our Introduction above) and has applications to several branches of applied mathematics, including queuing theory and the study of linear polymers.  The seminal work of  Bousquet-M{\'e}lou and Mishna~\cite{ BoMi10} gave a uniform approach to several enumerative questions, including the nature of a model's generating function (algebraic, D-finite, etc.) and the determination of exact or asymptotic counting formulas.  In particular, they used the \emph{orbit sum method} (in a manner similar to Section~\ref{sec:KernelMethod}) to prove that the generating functions corresponding to 22 of the 79 non-equivalent two dimensional models were D-finite (they conjectured that one additional model was D-finite and that the rest were not).  Around the same time, Bostan and Kauers~\cite{ BoKa09} used computer algebra approaches to guess differential equations satisfied by the generating functions of 23 models (the 22 proven D-finite by Bousquet-M{\'e}lou and Mishna and the one they conjectured to be D-finite) which they then used to guess dominant asymptotics for these models (see Table~\ref{tab:guessed}).  They later proved, using another computer algebra approach, that the 23$^\text{rd}$ conjectured D-finite model of Bousquet-M{\'e}lou and Mishna is in fact algebraic (and thus also D-finite)\footnote{This was later proved using non-computer based arguments by Bostan, Kurkova, and Raschel~\cite{BoKuRa17} and Bousquet-M{\'e}lou~\cite{Bousquet-Melou2016}.}.  Furthermore, 5 of the remaining 56 models were proven to admit non D-finite generating functions by Melczer and Mishna~\cite{ MeMi14a} and strong evidence for non D-finiteness of the final 51 generating functions has been provided by several sources~\cite{ KuRa12, BoRaSa14}. 

We now look at the application of the general formulas developed in Section~\ref{sec:Asym} to proving the guessed asymptotics in Table~\ref{tab:guessed}.

\begin{table}[ht]
\centering
\begin{tabular}{ | c | c @{ \hspace{0.01in} }@{\vrule width 1.2pt }@{ \hspace{0.01in} }  c | c @{ \hspace{0.01in} }@{\vrule width 1.2pt }@{ \hspace{0.01in} }  c | c |  }
  \hline
    $\mS$ & Asymptotics & $\mS$ & Asymptotics & $\mS$ & Asymptotics \\ \hline
  &&&&& \\[-5pt] 
  \diagr{N,S,E,W}  & $\frac4\pi \cdot \frac{4^n}n$ &
  \diagr{NE,NW,S}  & $\frac{\sqrt{3}}{2\sqrt{\pi}} \cdot \frac{3^n}{\sqrt{n}}$ &
  \diagr{N,SE,SW} & $\frac{A_n}{\pi} \cdot \frac{(2\sqrt{2})^n}{n^2}$  \\
  \diagr{NE,SE,NW,SW}  & $\frac2\pi \cdot \frac{4^n}n$ &
  \diagr{N,NW,NE,S}  & $\frac4{3\sqrt{\pi}} \cdot \frac{4^n}{\sqrt{n}}$ &
  \diagr{N,S,SE,SW} & $\frac{B_n}{\pi} \cdot \frac{(2\sqrt{3})^n}{n^2}$ \\
  \diagr{N,S,NE,SE,NW,SW} & $\frac{\sqrt{6}}\pi \cdot \frac{6^n}n$  &
  \diagr{N,NE,NW,SE,SW}  & $\frac{\sqrt{5}}{3\sqrt{2\pi}} \cdot \frac{5^n}{\sqrt{n}}$ &
  \diagr{NE,NW,SE,SW,S}  & $\frac{C_n}{\pi} \cdot \frac{(2\sqrt{6})^n}{n^2}$ \\
  \diagr{N,S,E,W,NW,SW,SE,NE}  & $\frac{8}{3\pi} \cdot \frac{8^n}n$ &
  \diagr{NE,NW,E,W,S}  & $\frac{\sqrt{5}}{2\sqrt{2\pi}} \cdot \frac{5^n}{\sqrt{n}}$ &
  \diagr{N,E,W,SE,SW}  & $\frac{\sqrt{8}(1+\sqrt{2})^{7/2}}{\pi} \cdot \frac{(2+2\sqrt{2})^n}{n^2}$ \\
  \diagr{NE,W,S}   & $\frac{2\sqrt{2}}{\Gamma(1/4)} \cdot \frac{3^n}{n^{3/4}}$ &
  \diagr{N,NW,NE,E,W,S} & $\frac{2\sqrt{3}}{3\sqrt{\pi}} \cdot \frac{6^n}{\sqrt{n}}$ &
  \diagr{N,E,W,S,SW,SE}  & $\frac{\sqrt{3}(1+\sqrt{3})^{7/2}}{2\pi} \cdot \frac{(2+2\sqrt{3})^n}{n^2}$  \\
  \diagr{N,E,SW}   & $\frac{3\sqrt{3}}{\sqrt2\Gamma(1/4)} \cdot \frac{3^n}{n^{3/4}}$ &
  \diagr{N,E,W,NE,NW,SE,SW} & $\frac{\sqrt{7}}{3\sqrt{3\pi}} \cdot \frac{7^n}{\sqrt{n}}$ & 
  \diagr{NE,NW,E,W,SE,SW,S}  & ${\scriptstyle \frac{\sqrt{570-114\sqrt{6}}(24\sqrt{6}+59)}{19\pi} \cdot \frac{(2+2\sqrt{6})^n}{n^2}}$ \\
  \diagr{N,NE,E,S,SW,W}   & $\frac{\sqrt{6\sqrt{3}}}{\Gamma(1/4)} \cdot \frac{6^n}{n^{3/4}}$ &
  \diagr{N,W,SE}   & $\frac{3\sqrt{3}}{2\sqrt{\pi}} \cdot \frac{3^n}{n^{3/2}}$ &
  \diagr{E,SE,W,NW}  & $\frac{8}{\pi} \cdot \frac{4^n}{n^2}$ \\
  \diagr{NE,E,SW,W}  & $\frac{4\sqrt3}{3\Gamma(1/3)} \cdot \frac{4^n}{n^{2/3}}$ &
  \diagr{NW,SE,N,S,E,W}   & $\frac{3\sqrt{3}}{2\sqrt{\pi}} \cdot \frac{6^n}{n^{3/2}}$  && \\
\hline
\end{tabular}
\caption{Asymptotics for the $23$ D-finite models of Bostan and Kauers.} \label{tab:guessed}
\vspace{-0.2in}
{\small
\[
A_n = \begin{cases} 24\sqrt{2} &: n \text{ even} \\ 32 &: n \text{ odd} \end{cases} ,  \quad
B_n = \begin{cases} 12\sqrt{3} &: n \text{ even} \\ 18 &: n \text{ odd} \end{cases} ,  \quad
C_n = \begin{cases} 12\sqrt{30} &: n \text{ even} \\ 144/\sqrt{5} &: n \text{ odd} \end{cases} \]
}
\end{table} 

\subsection{The Highly Symmetric Models}
\label{ss:2Dhighasm}
\[ \diagr{N,S,E,W} \qquad \diagr{NE,SE,NW,SW} \qquad \diagr{N,S,NE,SE,NW,SW} \qquad \diagr{N,S,E,W,NW,SW,SE,NE} \]
Four of the models in Table~\ref{tab:guessed} have step sets that are symmetric over every axis.  This means that their asymptotics follow directly from the work of Melczer and Mishna~\cite{ MelczerMishna2016} (see Theorem~\ref{thm:MeMiAsm} above). The asymptotic order is $|\mS|^n n^{-1}$ in each case.

\subsection{Positive Drift Models Missing One Symmetry}
\label{ss:2Dposasym}
\[ \diagr{NE,NW,S} \qquad \diagr{N,NW,NE,S} \qquad \diagr{N,NE,NW,SE,SW} \qquad \diagr{NE,NW,E,W,S} \qquad   \diagr{N,NW,NE,E,W,S} \qquad \diagr{N,E,W,NE,NW,SE,SW} \]
There are six models whose step sets are missing one symmetry and have positive drift; one can directly apply Theorem~\ref{thm:PosAsm} to prove the asymptotics listed. The asymptotic order is $|\mS|^n n^{-1/2}$ in each case.

\subsection{Negative Drift Models Missing One Symmetry}
\label{ss:2Dnegasym}
\[ \diagr{N,SE,SW} \qquad \diagr{N,S,SE,SW} \qquad \diagr{NE,NW,SE,SW,S} \qquad \diagr{N,E,W,SE,SW} \qquad \diagr{N,E,W,S,SW,SE}  \qquad \diagr{NE,NW,E,W,SE,SW,S} \]
There are six models whose step sets are missing one symmetry and have negative drift; one can apply Theorem~\ref{thm:NegAsm} to prove the asymptotics listed (we note that the original table of guessed asymptotics by Bostan and Kauers~\cite{BoKa09} has small errors in the constants for the first three of these models). The asymptotic order is $\left(Q(\bone)+2\sqrt{A(\bone)B(\bone)}\right)^n n^{-2}$ in each case.

\begin{example} Consider the model defined by step set 
\[ \mS = \{ (0,1),(-1,-1),(0,-1),(1,-1) \} = \{N,SE,S,SW\}.\]  
Here we have 
\begin{align*}
G(\bz,t) &= (1+x) \left(1 - y^2 (\ox + 1 + x)\right) / (1-y)  \\
H(\bz,t) &=  1-t(x + y^2+xy^2+x^2y^2),
\end{align*}
and two of the eight possible points described by Theorem~\ref{thm:contrib} are contributing points: $\bp_1 = (1,1/\sqrt{3},1/2)$ and $\bp_2 = (1,-1/\sqrt{3},1/2)$. Using Sage to implement Proposition~\ref{prop:NegAsm}, we can calculate the contribution at each contributing point to be
\begin{align*}
 \Psi^{(\bp_1)}_{n} &= \frac{3\sqrt{3}(2+\sqrt{3})}{\pi} \cdot \frac{(2\sqrt{3})^n}{n^2}\\
 \Psi^{(\bp_2)}_{n} &= \frac{3\sqrt{3}(2-\sqrt{3})}{\pi} \cdot \frac{(-2\sqrt{3})^n}{n^2}
\end{align*}
so that the number of walks of length $n$ satisfies
\begin{align*} 
s_n &= \frac{(2\sqrt{3})^n}{n^2} \cdot \frac{3\sqrt{3}}{\pi}\left(\sqrt{3}(1-(-1)^n) + 2(1+(-1)^n)  + O(n^{-1})\right) \\
&= 
\left\{ \begin{array}{ll}    
\frac{(2\sqrt{3})^n}{n^2} \cdot \left( \frac{12\sqrt{3}}{\pi} + O(n^{-1})\right) &: n \text{ even} \\ 
\frac{(2\sqrt{3})^n}{n^2} \cdot \left( \frac{18}{\pi} + O(n^{-1})\right) &: n \text{ odd} \end{array}  \right.
\end{align*}
\end{example}

The first three models here each have two contributing points that determine the dominant asymptotics, giving a periodicity in the coefficients as seen in the above example.  

\subsection{The Exceptional (Zero Orbit Sum) Algebraic Cases}
\label{ss:2Dexcasym}
\[ \diagr{NE,W,S} \qquad \diagr{N,E,SW} \qquad \diagr{N,NE,E,S,SW,W} \qquad \diagr{NE,E,SW,W}\]
There are four models for which the orbit sum method fails to give an expression for the walk generating functions as the diagonals of rational functions (meaning the techniques of analytic combinatorics in several variables as described above cannot by directly used).  Luckily, these four models are algebraic and explicit minimal polynomials for the generating functions are known: the first was found by Mishna~\cite{ Mish09}, the next two by Bousquet-M{\'e}lou and Mishna~\cite{ BoMi10} and the final model---known as Gessel's walk---is treated in Bostan and Kauers~\cite{ BoKa10}.  It is effective to determine the asymptotics of a sequence from its generating function's minimal polynomial (under gentle technical conditions), so the asymptotics for these cases follow rigorously through univariate methods (see Chapter VII.7 of Flajolet and Sedgewick~\cite{ FlSe09} and~\cite{ Mish09,BoKa10}). In fact, the multivariate generating function enumerating walks by length and endpoint is algebraic for each model, a stronger property. 

\subsection{The Remaining Three Models}
\label{ss:2remasym}
\[ \diagr{N,W,SE} \qquad \diagr{NW,SE,N,S,E,W} \qquad \diagr{E,SE,W,NW} \]
There are three models not covered by the above cases.  They do not exhibit any symmetries, but the orbit sum method still gives a rational diagonal expression. Asymptotics of these models were previous given by Bousquet-M{\'e}lou and Mishna~\cite{ BoMi10}: although their multivariate generating functions enumerating walks by length and endpoint are transcendental, the first two models have univariate generating functions $F(t)$ counting walks ending anywhere in the quadrant which are algebraic. The final model does not have an algebraic univariate generating function, but Bousquet-M{\'e}lou and Mishna determined asymptotics by exploiting the fact that the coefficients of its multivariate generating function are \emph{Gosper summable}; see~\cite[Proposition 11]{ BoMi10} for details.

\subsubsection{Case 1: \texorpdfstring{$\mS = \{N,W,SE\}$}{S={N,W,SE}}} 
Applying the kernel method, the counting generating function satisfies
\[ F(t) = \Delta \left( \frac{(x^2-y)(1-\ox\oy)(x-y^2)}{(1-x)(1-y)(1-xyt(\oy+y\ox+x))} \right). \]
Furthermore, the kernel method implies (see, for example, Bousquet-M{\'e}lou and Mishna~\cite{BoMi10} or Bostan et al.~\cite{BostanChyzakHoeijKauersPech2017}) that  
\[ F(t) = \frac{1-t-\sqrt{1-2t-3t^2}}{2t^2} \]
is algebraic, and asymptotics can be determined directly from this specification. Alternatively, one can perform a (more difficult) multivariate singularity analysis on this rational function. Although the rational function has smooth and transverse multiple points that are minimal and critical, we cannot directly apply the asymptotic methods discussed above at the point $(1,1,1/3)$---which turns out to be a contributing singularity---as the gradient of $H_1$ at that point is parallel to the gradient of the function $\phi(x,y,t) = \log (xyt)$ occurring in the Fourier-Laplace integral that must be analyzed to determine asymptotics. The theory in this case, where three factors in the denominator intersect, still needs to be fully developed. Here, one can write
\[  x^2-y = (x-1)(x+1)-(y-1)\]
to decompose the rational function as
\begin{align*}
\frac{(x^2-y)(1-\ox\oy)(x-y^2)}{(1-x)(1-y)(1-xyt(\oy+y\ox+x))} = &  - \frac{(1-\ox\oy)(x-y^2)(x+1)}{(1-y)(1-xyt(\oy+y\ox+x))} \\
& + \frac{(1-\ox\oy)(x-y^2)}{(1-x)(1-xyt(\oy+y\ox+x))},
\end{align*}
where each of the summands now contains only two factors in the denominator. After this simplification, a (still difficult) asymptotic analysis can be applied to determine the asymptotic contribution of each summand to the diagonal sequence. Using the explicit algebraic expression, or the multivariate approach, gives that the counting sequence for the number of walks on these steps satisfies
\[s_n = \frac{3^n}{n^{3/2}} \, \left( \frac{3\sqrt{3}}{2\sqrt{\pi}} + O\left(n^{-1}\right)\right). \]
It is interesting that the D-finite models which are harder to approach using multivariate analytic methods are precisely those which have algebraic generating functions; we do not see a deep reason why this should be true.

\subsubsection{Case 2: \texorpdfstring{$\mS =\{NW,SE,N,S,E,W\}$}{S={NW,SE,N,S,E,W}}} Applying the kernel method, we see that the counting generating function satisfies
\[ F(t) = \Delta \left( \frac{(x-y^2)(1-\ox\oy)(x^2-y)}{(1-x)(1-y)(1-txy(x+y+x\oy+y\ox+\ox+\oy))} \right). \]
Analogously to the last case, one can compute an algebraic expression
\[ F(t) = \frac{1-2t-\sqrt{1-4t-12t^2}}{8t^2} \]
for the generating function, which immediately gives asymptotics, or decompose the multivariate rational function into a sum of rational functions with simpler singular sets and perform a multivariate analysis. In either case, one obtains that the counting sequence for the number of walks on these steps satisfies
\[ s_n = \frac{6^n}{n^{3/2}} \, \left(\frac{3\sqrt{3}}{2\sqrt{\pi}} + O\left(n^{-1}\right)\right). \]

\subsubsection{Case 3: \texorpdfstring{$\mS = \{E,SE,W,NW\}$}{S={E,SE,W,NW}}} Applying the kernel method, we see that the counting generating function satisfies
\[ F(t) = \Delta \left( \frac{(x+1)(\ox^2-\oy)(x-y)(x+y)}{1-xyt(x+x\oy+y\ox+\ox)} \right). \]
This case turns out to be easy to analyze, since the denominator is smooth.  There are two points that satisfy the critical point equations: $\bp_1 = (1,1,1/4)$ and $\bp_2 = (-1,1,1/4)$, both of which are minimal and smooth.  As the numerator has a zero of order 2 at $\bp_1$ but order 3 at $\bp_2$, in fact only $\bp_1$ contributes to the dominant asymptotics.  The Sage package of Raichev---implementing Theorem~\ref{thm:multasm}---computes the contributions at both points and shows that the counting sequence for the number of walks on these steps satisfies
\[ s_n =  \frac{4^n}{n^2} \cdot \frac{8}{\pi} + O\left(\frac{4^n}{n^3}\right) \]
Weighted versions of this step set were studied in detail by Courtiel et al.~\cite{CourtielMelczerMishnaRaschel2017}.

\section{Further Considerations}
\label{sec:Conclusion}

We end with some additional remarks and generalizations.

\subsection{Weighted models}
Although our results hold for models whose steps have positive real weights, we have not yet given an example with positive weights not equal to one. We do so now.


\begin{example}
Consider the general 2D model with one symmetry defined by the following step set with non-negative real weights,

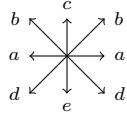
\begin{center}
\begin{tikzpicture}[scale=.5]
    \draw[->] (0,0) -- (-1,-1) node[left] {$\scriptstyle d$};
    \draw[->] (0,0) -- (-1,0) node[left] {$\scriptstyle a$};
    \draw[->] (0,0) -- (-1,1) node[left] {$\scriptstyle b$};
    \draw[->] (0,0) -- (0,-1) node[below] {$\scriptstyle e$};
    \draw[->] (0,0) -- (0,1) node[above] {$\scriptstyle c$};
    \draw[->] (0,0) -- (1,-1) node[right] {$\scriptstyle d$};
    \draw[->] (0,0) -- (1,0) node[right] {$\scriptstyle a$};
    \draw[->] (0,0) -- (1,1) node[right] {$\scriptstyle b$};
\end{tikzpicture}
\end{center}
\noindent
Then $s_n$ is asymptotic to
\[
\left(2(a+b+d)+c+e\right)^n n^{-1/2} \cdot \left[ \; \left(1 - \frac{2d+e}{2b+c}\right) \sqrt{\left(\frac{2(a+b+d)+c+e}{(a+b+d)\pi}\right)} \; \right] 
\]
when $2b+c > 2d+e > 0$ (positive drift), whereas for $0 < 2b+c < 2d+e$ (negative drift) and $a\neq 0$ it is asymptotic to 
\[
\left(2a+2\sqrt{(2b+c)(2d+e)}\right)^n n^{-2} \cdot C 
\]
with 
\[
C =  \left[\frac{(2a+2\sqrt{(2b+c)(2d+e)})^2}{2\pi \left(1 - \sqrt{\frac{2b+c}{2d+e}} \right)^2 \left((2b+c)(2d+e)\right)^{3/4} \sqrt{d\sqrt{\frac{2b+c}{2d+e}}+a+b\sqrt{\frac{2d+e}{2b+c}}}} \right].
\]
Asymptotics in the subcase when $0 < 2b+c < 2d+e$ (negative drift) and $a = 0$ similarly follows from Theorem~\ref{thm:NegAsm}, resulting in an unwieldy formula due to periodicity. Finally, when $2b+c=2d+e$ (zero drift) and $b=d$ and $c=e$ (highly symmetric) then $s_n$ is asymptotic to
\[ \left(2a+2c+4b\right)^n n^{-1} \cdot \left[ \;\frac{2a+2c+4b}{\pi \sqrt{(a+2b)(c+2b)}} \; \right]  . \]
The non-highly symmetric zero drift cases are outside the scope of our results. Conjecturally, they have dominant asymptotics which are a constant times $(2a+2b+2c+2d)^n n^{-1}$ for generic $a,b,c,d$, although there are values of the parameters for which this does not hold.
\end{example}

\subsection{Decidability of asymptotics}
\label{sec:computability} 
The techniques of analytic combinatorics in several variables are currently at the front line of research into computability questions in enumerative combinatorics. Given a univariate rational generating function, or an algebraic generating function encoded by its minimal polynomial and a sufficient number of initial terms, there are algorithms that take the function and return asymptotics of its power series coefficients at the origin. On the other hand, it is an open problem whether it is decidable to take a D-finite generating function encoded by an annihilating linear differential equation and initial conditions and determine asymptotics of its counting sequence. In slightly restricted settings (for instance, when the D-finite generating function has integer coefficients and positive radius of convergence) a careful singularity analysis allows one to determine a so-called \emph{asymptotic basis}: a finite collection of terms $\Delta_1,\dots,\Delta_d$ with asymptotic expansions of the form
\[ \Delta_j = \rho_j^n n^{\alpha_j} (\log n)^{\kappa_j} \left( C_0^{(j)} + \frac{C_1^{(j)}}{n} + \cdots \right) \]
that can be determined explicitly to any finite order, such that asymptotics of the coefficient sequence $c_n$ is an $\mathbb{R}-$linear combination of the $\Delta_j$,
\[ c_n \sim K_1 \, \Delta_1 + \cdots + K_r \, \Delta_r, \qquad K_j \in \mathbb{R}. \]
See Flajolet and Sedgewick~\cite[Sec VII. 9]{FlSe09} for details. 

In this way, decidability of asymptotics can be reduced to the determination of the \emph{connection coefficients} $K_j$. If, without loss of generality, asymptotics of $\Delta_1$ dominates asymptotics of the other $\Delta_j$ then $\Delta_1$ typically determines (up to a scaling multiple) asymptotics of $c_n$. However, if the constant $K_1$ is zero then $c_n$ can have drastically different asymptotic behaviour than $\Delta_1$. Determining the coefficients $K_j$ is known as the \emph{connection problem}.

Because the class of multivariate rational diagonals contains the class of algebraic functions, and is contained in the class of D-finite functions, the techniques of analytic combinatorics in several variables offer tools to investigate the connection problem (see Melczer~\cite{Melczer2017} for an in-depth look at this approach). For instance, Bostan et al.~\cite{BostanChyzakHoeijKauersPech2017} give annihilating differential equations for each lattice path generating function in Table~\ref{tab:guessed}, even representing them in terms of explicit hypergeometric functions; however, they were not able to prove all asymptotics in that table, because of the connection problem.  For instance, they show~\cite[Conjecture 2]{BostanChyzakHoeijKauersPech2017} that the number of walks with step set $\mS = \{(0,-1),(-1,1),(1,1)\}$ has dominant asymptotics of the form $\frac{\sqrt{3}}{2\sqrt{\pi}} 3^k k^{-1/2}$ if and only if the integral
\begin{align*} 
I := \int_0^{1/3} & \left\{  \frac{(1-3v)^{1/2}}{v^3(1+v^2)^{1/2}}\left[ 1 
+ (1-10v^3) \cdot {}_2F_1\left(\left. \genfrac{}{}{0pt}{}{3/4, 5/4}{1} \right| 64v^4 \right)  \right. \right. \\
& \hspace{1.1in} \left.\left. + \, 6v^3(3-8v+14v^2) \cdot {}_2F_1\left(\left. \genfrac{}{}{0pt}{}{5/4, 7/4}{2} \right| 64v^4 \right)\right] \right. \\
& \hspace{3.4in} \left. - \frac{2}{v^3} + \frac{4}{v^2} \right\} dv 
\end{align*}
has the value $I=1$ (see that paper for details on the notation used).  Using the multivariate singularity analysis discussed above, we are able to circumvent these difficulties, and resolve the connection problem for these lattice path models. As an indirect corollary of our asymptotic results, we thus determine the values of certain complicated integral expressions involving hypergeometric functions.

\subsection{Walks returning to boundaries}
The kernel method as presented here uses the multivariate generating function $F(\bz,t)$ tracking walk length and endpoint to derive a rational diagonal expression for the univariate generating function $F(\bone,t)$ counting the number of walks ending anywhere. Also of interest is the number of walks ending on one or more of the boundary hyperplanes in the first orthant; if $V \subset \{1,\dots,d\}$ then 
\[ F(\bz,t){\Big|}_{\substack{z_j = 0, j \in V \\ z_j = 1, j \notin V}} \]
counts the number of walks returning to the intersection of the boundary hyperplanes $\{z_j=0\}$ for $j \in V$. Lemma~\ref{lem:diagsubs} can easily be generalized to the following.

\begin{lemma}
Let $P(\bz,t) \in \mathbb{Q}[z_1,\oz_1,\dots,z_d,\oz_d][[t]] \subset \mR$.  Then 
\[ \left([\bz^\geq]P(\bz,t)\right) {\Big|}_{\substack{z_j = 0, j \in V \\ z_j = 1, j \notin V}} = \Delta \left(\frac{P\left(\oz_1,\dots,\oz_d,z_1\cdots z_d\cdot t\right)}{(1-z_1)\cdots(1-z_d)} \cdot \prod_{j \in V}(1-z_j) \, \right). \]
\end{lemma} 

Thus, following the arguments above, if $Q(\bz,t) = G(\bz,t)/H(\bz,t)$ is the rational function given in Theorem~\ref{thm:diag2} such that $F(\bone,t)=(\Delta Q)(t)$ then
\[ F(\bz,t){\Big|}_{\substack{z_j = 0, j \in V \\ z_j = 1, j \notin V}} = \Delta \left( Q(\bz,t) \cdot \prod_{j \in V}(1-z_j) \,  \right). \]

This close link between the diagonal expressions for walks ending anywhere and walks ending on boundary hyperplanes allows us to reuse much of the work above to derive asymptotics for walks ending on boundary hyperplanes. In particular, if $V$ does not contain $d$ then the singular sets of both multivariate rational functions obtained are the same, so the contributing points calculated by Theorem~\ref{thm:contrib} are still contributing. Analysis of asymptotics is easy for any fixed model, but the additional zeros in the numerator of $Q(\bz,t) \cdot \prod_{j \in V}(1-z_j)$ at contributing points make explicit expressions for generic models harder to calculate. 

When $V$ contains $d$, then the factor of $1-z_d$ in the numerator of $Q(\bz,t)$ will cancel the new factor $1-z_d$ in the numerator. In the negative drift and zero drift cases this has no bearing on the contributing singularities, and hence on the  exponential growth of the number of walks returning to the hyperplane $\{z_d=0\}$. However in the positive drift case the contributing singularities will change and the exponential growth will be smaller for walks returning to the hyperplane $\{z_d=0\}$ than for general walks.

Using the Sage package of Raichev to compute asymptotic contributions, Table~\ref{tab:BoundaryAsm} gives asymptotics for the number of walks returning to one or both of the boundary axes on the 2D quadrant models analyzed above, where

{\footnotesize
\[
\delta_n = \begin{cases} 1&: n \equiv 0 \mod 2 \\ 0 &: \text{otherwise} \end{cases} \qquad
\sigma_n = \begin{cases} 1&: n \equiv 0 \mod 3 \\ 0 &: \text{otherwise} \end{cases} \qquad
\epsilon_n = \begin{cases} 1&: n \equiv 0 \mod 4 \\ 0 &: \text{otherwise} \end{cases}
\]
}

\noindent
and
{\footnotesize 
\[ 
\gamma_n = \begin{cases} 
448\sqrt{2}&: n \equiv 0 \mod 4 \\
640&: n \equiv 1 \mod 4 \\
416\sqrt{2}&: n \equiv 2 \mod 4 \\
512&: n \equiv 3 \mod 4 
\end{cases}
\]
}

\noindent
help account for periodicities that appear, and the algebraic constants $A,B,$ and $C$ are given by

{\footnotesize \[A = (156+41\sqrt{6})\sqrt{23-3\sqrt{6}}, \quad B = (583+138\sqrt{6})\sqrt{23-3\sqrt{6}}, \quad C = (4571+1856\sqrt{6})\sqrt{23-3\sqrt{6}}. \]}

\noindent
This completes, for the first time, a proof of conjectured asymptotics given by Bostan et al.~\cite{BostanChyzakHoeijKauersPech2017}. Note that the second and third columns show that periodicity can occur even with positive drift models (unlike the situation for walks ending anywhere analyzed in previous sections).
 
\begin{table}[ht!]
{\footnotesize
\centering
\begin{tabular}{ | c | c | c | c |}
  \hline
  $\mS$ & Return to $x$-axis & Return to $y$-axis & Return to origin \\ \hline
  &&& \\[-5pt]
  \diagr{N,S,E,W}  & $\frac{8}{\pi} \cdot \frac{4^n}{n^2}$ & 
  $\frac{8}{\pi} \cdot \frac{4^n}{n^2}$ &  
  $\delta_n \frac{32}{\pi} \cdot \frac{4^n}{n^3}$ \\[+2mm]
  \diagr{NE,SE,NW,SW}  & $\delta_n \frac{4}{\pi} \cdot \frac{4^n}{n^2}$ & 
  $\delta_n \frac{4}{\pi} \cdot \frac{4^n}{n^2}$ &  
  $\delta_n \frac{8}{\pi} \cdot \frac{4^n}{n^3}$ \\[+2mm]
  \diagr{N,S,NE,SE,NW,SW} &$\frac{3\sqrt{6}}{2\pi} \cdot \frac{6^n}{n^2}$ & 
  $\delta_n \frac{2\sqrt{6}}{\pi} \cdot \frac{6^n}{n^2}$ &  
  $\delta_n \frac{3\sqrt{6}}{\pi} \cdot \frac{6^n}{n^3}$ \\[+2mm]
  \diagr{N,S,E,W,NW,SW,SE,NE}  & $\frac{32}{9\pi} \cdot \frac{8^n}{n^2}$ & 
  $\frac{32}{9\pi} \cdot \frac{8^n}{n^2}$ &  
  $\frac{128}{27\pi} \cdot \frac{8^n}{n^3}$  \\[+2mm]
  \diagr{NE,NW,S}  & $\frac{3\sqrt{3}}{4\sqrt{\pi}}  \frac{3^n}{n^{3/2}}$ & 
  $\delta_n \frac{4\sqrt{2}}{\pi}  \frac{(2\sqrt{2})^n}{n^2}$ & 
  $\epsilon_n \frac{16\sqrt{2}}{\pi} \frac{(2\sqrt{2})^n}{n^3}$ \\[+2mm]
  \diagr{N,NW,NE,S}  & $\frac{8}{3\sqrt{\pi}} \frac{4^n}{n^{3/2}}$ & 
  $\delta_n \frac{4\sqrt{3}}{\pi}  \frac{(2\sqrt{3})^n}{n^2}$ & 
  $\delta_n \frac{12\sqrt{3}}{\pi}  \frac{(2\sqrt{3})^n}{n^3}$ \\[+2mm]
  \diagr{NE,NW,E,W,S} & $\frac{5\sqrt{10}}{16\sqrt{\pi}}  \frac{5^n}{n^{3/2}}$  & 
  $\frac{\sqrt{2}(1+\sqrt{2})^{3/2}}{\pi}  \frac{(2+2\sqrt{2})^n}{n^2}$ & 
  $\frac{2(1+\sqrt{2})^{3/2}}{\pi}  \frac{(2+2\sqrt{2})^n}{n^3}$ \\[+2mm]
  \diagr{N,NE,NW,SE,SW}  & $\frac{5\sqrt{10}}{24\sqrt{\pi}}  \frac{5^n}{n^{3/2}}$ & 
  $\delta_n \frac{4\sqrt{30}}{5\pi}  \frac{(2\sqrt{6})^n}{n^2}$ & 
  $\delta_n \frac{24\sqrt{30}}{25\pi}  \frac{(2\sqrt{6})^n}{n^3}$ \\[+2mm]
  \diagr{N,NW,NE,E,W,S} & $\frac{\sqrt{3}}{\sqrt{\pi}} \frac{6^n}{n^{3/2}}$ &
  $\frac{2\sqrt{3}(1+\sqrt{3})^{3/2}}{3\pi} \frac{(2+2\sqrt{3})^n}{n^2}$ & 
  $\frac{2(1+\sqrt{3})^{3/2}}{\pi} \frac{(2+2\sqrt{3})^n}{n^3}$ \\[+2mm]
  \diagr{N,E,W,NE,NW,SE,SW} & 
  $\frac{7\sqrt{21}}{54\sqrt{\pi}} \frac{7^n}{n^{3/2}}$ & 
  $ \frac{A}{285\pi} \frac{(2+2\sqrt{6})^n}{n^2}$ & 
  $\frac{2B}{1805\pi} \frac{(2+2\sqrt{6})^n}{n^3}$ \\[+2mm]
  \diagr{N,SE,SW}  & $\gamma_n \cdot \frac{1}{9\pi} \cdot \frac{(2\sqrt{2})^n}{n^3}$ & 
  $\delta_n \frac{4\sqrt{2}}{\pi}  \frac{(2\sqrt{2})^n}{n^2}$ & 
  $\epsilon_n \frac{16\sqrt{2}}{\pi} \frac{(2\sqrt{2})^n}{n^3}$ \\[+2mm]
  \diagr{S, SW, SE, N}  & $\left(\delta_n \frac{36\sqrt{3}}{\pi} + \delta_{n-1} \frac{54}{\pi}\right)\cdot \frac{(2\sqrt{3})^n}{n^3}$ & 
  $\delta_n \frac{4\sqrt{3}}{\pi}  \frac{(2\sqrt{3})^n}{n^2}$ & 
  $\delta_n \frac{12\sqrt{3}}{\pi}  \frac{(2\sqrt{3})^n}{n^3}$ \\[+2mm]
  \diagr{SE, SW, E, W, N} & $\frac{4(1+\sqrt{2})^{7/2}}{\pi} \cdot \frac{(2+2\sqrt{2})^n}{n^3}$  & 
  $\frac{\sqrt{2}(1+\sqrt{2})^{3/2}}{\pi}  \frac{(2+2\sqrt{2})^n}{n^2}$ & 
  $\frac{2(1+\sqrt{2})^{3/2}}{\pi}  \frac{(2+2\sqrt{2})^n}{n^3}$  \\[+2mm]
  \diagr{S, SE, SW, NE, NW}  & $\left( \delta_n \frac{72\sqrt{30}}{5\pi} + \delta_{n-1}\frac{864\sqrt{5}}{25\pi} \right) \cdot \frac{(2\sqrt{6})^n}{n^3}$ & 
  $\delta_n \frac{4\sqrt{30}}{5\pi}  \frac{(2\sqrt{6})^n}{n^2}$ & 
  $\delta_n \frac{24\sqrt{30}}{25\pi}  \frac{(2\sqrt{6})^n}{n^3}$ \\[+2mm]
  \diagr{S, SW, SE, E, W, N} & $\frac{3(1+\sqrt{3})^{7/2}}{2\pi} \cdot \frac{(2+2\sqrt{3})^n}{n^3}$ &
  $\frac{2\sqrt{3}(1+\sqrt{3})^{3/2}}{3\pi} \frac{(2+2\sqrt{3})^n}{n^2}$ & 
  $\frac{2(1+\sqrt{3})^{3/2}}{\pi} \frac{(2+2\sqrt{3})^n}{n^3}$  \\[+2mm]
  \diagr{NE, NW, E, W, SE, SW, S} & 
  $\frac{6C}{1805\pi} \cdot \frac{(2+2\sqrt{6})^n}{n^3}$ & 
  $\frac{A}{285\pi} \frac{(2+2\sqrt{6})^n}{n^2}$ &
  $\frac{2B}{1805\pi} \frac{(2+2\sqrt{6})^n}{n^3}$ \\[+2mm]
  \diagr{N,W,SE}   & $\frac{27\sqrt{3}}{8\sqrt{\pi}} \cdot \frac{3^n}{n^{5/2}}$ & 
  $\frac{27\sqrt{3}}{8\sqrt{\pi}} \cdot \frac{3^n}{n^{5/2}}$ & 
  $\sigma_n \frac{81\sqrt{3}}{\pi} \cdot \frac{3^n}{n^4}$ \\[+2mm]
  \diagr{E,SE,W,NW}  & $\delta_n \frac{32}{\pi} \cdot \frac{4^n}{n^3}$ & 
  $\frac{32}{\pi} \cdot \frac{4^n}{n^3}$ & 
  $\delta_n \frac{768}{\pi} \cdot \frac{4^n}{n^5}$ \\[+2mm]
  \diagr{NW,SE,N,S,E,W} & $\frac{27\sqrt{3}}{8\sqrt{\pi}} \cdot \frac{6^n}{n^{5/2}}$ & 
  $\frac{27\sqrt{3}}{8\sqrt{\pi}} \cdot \frac{6^n}{n^{5/2}}$ & 
  $\frac{27\sqrt{3}}{\pi} \cdot \frac{6^n}{n^4}$ \\[+2mm]
  \hline
\end{tabular}
\caption{Asymptotics of quadrant walks returning to the $x$-axis, the $y$-axis, and the origin, respectively.}
\label{tab:BoundaryAsm}
}
\end{table}

\subsection{Zero drift models}
\label{sec:zerodrift}
In the non-highly symmetric zero drift case, when $A(\bone) = B(\bone)$ but $A(\bzht{d}) \neq B(\bzht{d})$,  Theorem~\ref{thm:contrib} implies that we can have contributions from the point $\bp:=\bp_1=\bp_2$ on the stratum $\mV_1 \cap \mV_3$, possibly with other points lying on locally smooth parts of $\mV_1$.  Note that the numerator vanishes to at least first order at every critical point, and that this case cannot occur for unweighted steps in dimension $2$ (where every zero drift model is highly symmetric).

Since $\bp$ is on the intersection of $\mV_1$ and $\mV_3$, and the numerator vanishes, we expect it to give an asymptotic contribution of $C\cdot |\mS|^n \cdot n^{-d/2-1/2}$, while the other (locally smooth) contributing points have a contribution of $O\left(|\mS|^n \cdot n^{-d/2-1}\right)$.  Thus, if we can determine a second order contribution at $\bp$ and show that it does not vanish, we will have found dominant asymptotics.

Asymptotic contributions of minimal critical points are determined by analyzing integrals of the form
\[ \int_{[-1,1]^r} A(\bt) e^{-n \phi(\bt)} \, d\bt \]
where $r \in \mathbb{N}$, and $A$ and $\phi$ analytic functions from $[-1,1]^r$ to $\mathbb{C}$ (see~\cite{PeWi13} for details). When the gradient of $\phi$ vanishes in the interior of $[-1,1]^r$, and other technical conditions on $A$ and $\phi$ (which are satisfied here) hold, the asymptotic formulas in Theorem~\ref{thm:multasm} follow. Unfortunately, in the non-highly symmetric zero drift case the gradient of $\phi$ vanishes on the boundary of the domain of integration, meaning the relevant asymptotic constants are not the same as those in Theorem~\ref{thm:multasm}. In fact, general asymptotics for such a situation have not yet been worked out in the context of ACSV.

Furthermore, while non-vanishing of the second order contribution at $\bp$ happens generically, there are models where vanishing does occur and finding dominant asymptotics requires a detailed analysis at several contributing singularities. Because of these added difficulties, including a need to extend the underlying analytic theory, a more nuanced study of the zero drift models will be the subject of future work.

\subsection{Connecting analytic and combinatorial behaviour}
As we have seen, the kernel method shows how nice combinatorial properties of a step set (like symmetry over axes) correspond to nice analytic properties of a multivariate rational function (like a singular set defined as the union of a small number of smooth manifolds) encoding the corresponding generating function. Furthermore, it is possible to turn this around: because diagonal sequences of multivariate rational functions with `simple' geometry at contributing singularities can only capture a restricted set of asymptotic behaviour, certain step sets whose asymptotics are sufficiently complicated cannot have their generating functions encoded as the diagonals of `nice' rational functions. 

The connection between analytic and combinatorial behaviour also helps explain patterns in asymptotics. For instance, it was previously observed that the exponential growth of 2D quadrant walks ending anywhere and the exponential growth of walks ending at the origin was the same for negative drift models but different for positive drift models. The strong connection between the diagonal representations of the corresponding generating functions explains why this is the case.

We believe that the tools of analytic combinatorics in several variables have much to offer the immensely popular area of lattice path enumeration, and hope that others will pick up and utilize the tools discussed here.

\section{Acknowledgements}
The authors thank Mireille Bousquet-M{\'e}lou for her generous contribution of Theorem~\ref{thm:diag2} to this article, which simplified the presentation substantially, and the anonymous referees for their insightful suggestions and comments.

\bibliographystyle{plain}
\bibliography{bibl}

\appendix
\section{Calculus Computations}
\label{appendixA} 

\subsection*{General Results}
Here we collect the results and proofs that are necessary for determining asymptotics but theoretically uninteresting. Our first two results will be useful for calculating derivatives.

\begin{lemma}
\label{lem:ABsimp}
Let $P, Q$ be smooth functions from $I \subset \mathbb{R}^d$ to $\mathbb{C}$, and suppose that $\bzer$ lies in the interior of $I$ and $Q(\bzer) \neq 0$. Let $\partial:=\partial_k$ be a partial derivative operator such that $(\partial P)(\bzer) = 0 = (\partial Q)(\bzer)$. 

Then $\partial(P/Q)(\bzer) = 0$ and
$$
\partial^2(P/Q)(\bzer) = \frac{Q(\bzer)\partial^2P(\bzer) - P(\bzer)\partial^2Q(\bzer)}{Q(\bzer)^2}.
$$
When 
\begin{align*}
P(\bt) &= \left(e^{i\theta_k} + e^{-i\theta_k}\right) P_k(\btht{k}) + R_k(\btht{k})\\
Q(\bt) &= \left(e^{i\theta_k} + e^{-i\theta_k}\right) Q_k(\btht{k}) + R'_k(\btht{k})
\end{align*}
then $(\partial P)(\bzer) = 0 = (\partial Q)(\bzer)$ and
$$
\partial^2(P/Q)(\bzer) = \frac{-2P_k(\bzer)Q(\bzer)+2P(\bzer)Q_k(\bzer)}{Q(\bzer)^2}.
$$
\end{lemma}
\begin{proof}
The first assertion follows from expanding the second derivative using the quotient rule, and applying the obvious simplification. The second follows from the first by direct substitution, since $(\partial^2 P)(\bzer)$ simplifies to $-2P_k(\bzer)$ and similarly for $Q$.
\end{proof}

\begin{lemma} \label{lem:partialS}
For a point $\bp \in \mathbb{C}^d$, let 
$$
\tilde{S}(\bt) :=  \tilde{\overline{S}}(\bt) = S(p_1e^{i\theta_1}, \dots, \overline{p_d}e^{-i\theta_d}).
$$
Then if $\bp_{\hat{d}} \in \{\pm1\}^{d-1}$ and $p_d = \pm \sqrt{B(\bp_{\hat{d}})/A(\bp_{\hat{d}})}$, we have for $1\leq j\leq k \leq d$:
$$
\partial_j\tilde{S}(\bzer) = 0
$$
and 
\[ \partial_j\partial_k\tilde{S}(\bzer) = \begin{cases}
       0 & :  \qquad \text{if $j \neq k$;} \\
       -2 B_d(\bp_{\hat{d}})/p_d &: \qquad \text{if $j = k = d$;}  \\
       -2 p_j B_j(\bp_{\hat{j}}) & :  \qquad \text{if $j=k < d$}.
   \end{cases}
   \]
Furthermore, if $j < d$ then
$$
\partial_d\partial_j^2\tilde{S}(\bzer) = -2i p_j \left(p_d \tilde{A}'_j(\bzer) - p_d^{-1} \tilde{B}'_j(\bzer) \right).
$$
If $p_d' = i \, p_d$ is the imaginary number corresponding to $p_d$, and $S(\bone,1) = S(\bp_{\hat{d}},p_d')$, then the values of all partial derivatives above equal the values of the derivatives calculated for $(\bp_{\hat{d}},p_d')$, potentially up to sign.
\end{lemma}

\begin{proof}
For $j<d$, applying $\partial_j$ to 
$$
\tilde{S}(\bt) = (p_je^{i\theta_j}+p_j^{-1}e^{-i\theta_j})\tilde{B_j}(\bt) + \tilde{Q_j}(\bt)
$$
yields
$$
\partial_j \tilde{S}(\bt) = (ip_je^{i\theta_j}-ip_j^{-1}e^{-i\theta_j})\tilde{B_j}(\bt).
$$
Note that from this point, applying any $\partial_k$ with $k\neq j$ will not change the factor $ip_je^{i\theta_j}-ip_j^{-1}e^{-i\theta_j}$. Also, evaluating at zero 
gives $i(p_j-p_j^{-1}) = 0$.

Repeating this with higher powers of $\partial_j$ gives a formula that is periodic in the exponent, with period $4$. In particular when we evaluate at $\bzer$ we obtain
$$
\partial_j^n \tilde{S}(\bzer) = 
\begin{cases}
(-1)^{n/2} 2p_j \tilde{B_j}(\bzer) & \qquad \text{if $n$ is even} \\
0 & \qquad\text{if $n$ is odd}.
\end{cases}
$$
A similar computation with $j=d$ yields 
$$
\partial_d^n \tilde{S}(\bzer) = 
\begin{cases}
(-1)^{n/2} 2 \tilde{B}(\bzer)/p_d & \qquad \text{if $n$ is even} \\
0 & \qquad\text{if $n$ is odd}.
\end{cases}
$$
Finally, consider the third order derivative $\partial_d\partial_j^2 \tilde{S}$. Writing 
$$
\tilde{S}(\bt) = p_de^{i\theta_d}A(\bt) + Q + p_d^{-1} e^{-i\theta_d} B(\bt)
$$
and differentiating using the formulae in Definition~\ref{def:Bpj} yields
the stated result.

The statement about $p_d'$ follows from the same considerations, using the fact that if $S(\bone,1) = S(\bp_{\hat{d}},p_d')$ then $B(\bp_{\hat{d}},p_d')/p_d'$ and $A(\bp_{\hat{d}},p_d') \, p_d'$ have the same argument.
\end{proof}

\subsection*{Negative Drift Model Calculations}

We now show that the quantities appearing in Theorem~\ref{thm:multasm} simplify for us. Since in our situation of interest we always have $L_0(\tilde{u},\tilde{g}) = \tilde{u}(\bzer) = 0$, we begin by considering the term corresponding to $k=1$ in \eqref{eqn:Psi}. For possible independent interest we show that some simplification is possible even in the general case.

\begin{lemma}
\label{lem:horm_simp}
In the general smooth case 
\[ L_1(\tilde{u}, \underline{\tilde{g}}) = -\frac{1}{2}\left( \mathcal{H}(\tilde{u})(\bzer) + \frac{\mathcal{H}^2(\tilde{u}\underline{\tilde{g}})(\bzer)}{4} + \frac{\tilde{u}\mathcal{H}^3(\underline{\tilde{g}}^2)(\bzer)}{48} \right). 
\] 
If $\tilde{u}$ vanishes to order at least $1$ at $\bzer$, then
\[ L_1(\tilde{u}, \underline{\tilde{g}})  = -\frac{1}{2}\left( \mathcal{H}(\tilde{u})(\bzer) + \frac{\mathcal{H}^2(\tilde{u}\underline{\tilde{g}})(\bzer)}{4} \right)
\] 
and only terms involving third partial derivatives of $\tilde{g}$ contribute to the $\mathcal{H}^2$ term.

If $\tilde{u}$ vanishes to order at least $2$ at $\bzer$ then
\[ L_1(\tilde{u}, \underline{\tilde{g}})  = -\frac{1}{2} \mathcal{H}(\tilde{u})(\bzer).
\] 

If $\tilde{u}$ vanishes to order at least $3$ at $\bzer$ then 
\[  L_1(\tilde{u}, \underline{\tilde{g}})  = 0.
\] 
\end{lemma}

\begin{proof}
First note that any partial derivative of order at most $2$ of $\underline{\tilde{g}}$ is zero when evaluated at $\bzer$, by construction. Furthermore all derivatives of $\underline{\tilde{g}}$ of degree more than $2$ yield the same result when evaluated at $\bzer$ as the corresponding derivative of $\tilde{g}$, since the difference between the two functions is quadratic. Thus in each nonzero term in an expansion of $L_1$ we may replace $\underline{\tilde{g}}$ by $\tilde{g}$.

Since $\underline{\tilde{g}}$ vanishes to order $3$ at $\bzer$, the term involving $\mathcal{H}^3$ simplifies substantially, since in order to obtain a nonzero term all the $6$th partial derivatives must be applied to $\underline{\tilde{g}}^2$ and so 
$\mathcal{H}^3(\tilde{u}\underline{\tilde{g}}^2)$ simplifies to $\tilde{u}\mathcal{H}^3(\underline{\tilde{g}}^2)$.

Similarly, $\mathcal{H}^2(\tilde{u}\underline{\tilde{g}})$ simplifies, since each 4th order partial derivative, when applied to the product $\tilde{u}\underline{\tilde{g}}$ and then evaluated at $\bzer$, only yields a nonzero result when at least $3$ of the derivations are applied to $\tilde{g}$. If $\tilde{u}$ vanishes to order $2$ then even these terms are zero. If $\tilde{u}$ vanishes to order $1$ then it is exactly the 3rd partials of $\tilde{g}$ that can contribute.
\end{proof}

This, combined with Theorem~\ref{thm:multasm}, directly gives the following.

\begin{proposition} \label{prop:NegAsm}
Let $\mS$ be a step set that is symmetric over all but one axis and takes a step forwards and backwards in each coordinate, and let $W$ be the set of contributing points determined by Theorem~\ref{thm:contrib}.  If $\mS$ has negative drift, then the number of walks of length $n$ that never leave the non-negative orthant satisfies
\begin{equation} s_n = \sum_{\bp \in W} \Psi^{(\bp)}_{n} \end{equation}
for
\[ \Psi^{(\bp)}_{n} = (p_1\cdots p_dp_t)^{-n} \left[ n^{-d/2-1} K_\bp C_\bp + O(n^{-1})\right], \]
where
\begin{align*}
K_\bp & = 2^{-d} \pi^{-d/2} \oS(\bp)^{d/2} \left(p_1B_1(\bp_{\hat{1}})\cdots p_{d-1}B_{d-1}(\bp_{\widehat{d-1}}) B_d(\bp_{\hat{d}}) / p_d \right)^{-1/2}  ,\\
C_{\bp} & = -\frac{1}{2}\left( \mathcal{H}(\tilde{u})(\bzer) + \frac{\mathcal{H}^2(\tilde{u}\underline{\tilde{g}})(\bzer)}{4} \right) \\
\intertext{for differential operator}
\mathcal{H} & = -\frac{\oS(\bp)}{2} \left( \frac{p_d}{B_d(\bp_{\hat{d}})}\partial_{d}^2 + \sum_{j < d} \frac{1}{p_jB_j(\bp_{\hat{j}})}\partial_{j}^2 \right) \\
\intertext{and} 
\tilde{u}(\bt) &= (1+p_1e^{i\theta_1}) \cdots (1+p_{d-1}e^{i\theta_{d-1}}) \left(1 - p_d^2e^{2i\theta_d} \frac{A\left( \bp_{\hat{d}}e^{i\bt_{\hat{d}}}\right)}{B\left(\bp_{\hat{d}}e^{i\bt_{\hat{d}}}\right)}\right) (1-p_de^{i\theta_d})^{-1}.\\
\end{align*}
\end{proposition}

We now show that the derivatives of $\tilde{g}$ and $\tilde{u}$ simplify substantially, giving Theorem~\ref{thm:NegAsm}.

\begin{proposition}
\label{prop:NegAsm2}
In the situation of Proposition~\ref{prop:NegAsm}, we have 
$$
C_{\bp} = \frac{\overline{S}(\bp)\prod_{j<d} (1+p_j)}{1-p_d} \left[ \frac{1}{A(\bp)p_d(1-p_d)} + \sum_{j=1}^{d-1} \frac{1-p_j}{2p_jB_j(\bp)}\left(\frac{A'_j(\bp)}{A(\bp)} - \frac{B'_j(\bp)}{B(\bp)} \right)\right].
$$
\end{proposition}

\begin{proof}
Note that $\partial_k \tilde{g} = - \partial_k \St/\St$. This evaluates to zero at $\bzer$. It follows from Lemma~\ref{lem:ABsimp} that $\partial_k^n \tilde{g}$ evaluates at $\bzer$ to $ - \partial_k^n\St(\bzer)/\St(\bzer)$. Also, when we evaluate $\partial_d\partial_j^2 \tilde{g}$  at $\bzer$, it simplifies to $\partial_d\partial_j^2 \tilde{S}(\bzer)/\tilde{S}(\bzer)$.

Now  define
\begin{align*}
X:= & \prod_{j<d} (1+p_je^{i\theta_j})\\
Y:= & 1 - p_d^2e^{2i\theta_d}\frac{A(\bpht{d}e^{i\bt_{\hat{d}}})}{B(\bpht{d}e^{i\bt_{\hat{d}}})} \\
Z:= & (1 - p_de^{i\theta_d})^{-1}
\end{align*}
so that
$$
\tilde{u} = XYZ.
$$

We first seek to compute
\[
 -\frac{1}{2} \mathcal{H}(\tilde{u})(\bzer) = \frac{\overline{S}(\bp)}{4}  \left( \frac{p_d}{B_d(\bp_{\hat{d}})}\partial_{d}^2 \tilde{u}(\bzer) + \sum_{j < d} \frac{1}{p_jB_j(\bp_{\hat{j}})}\partial_{j}^2 \tilde{u}(\bzer)  \right).
\]
When $k<d$, we have $\partial_k^2 \tilde{u} = \partial^2_k (XYZ)$.
Expanding via the product rule and evaluating at $\bt = \bzer$ we see that each term with $Y$ as a factor yields zero
because $Y$ vanishes at at $\bp$, and each term with $\partial_kY$ as a factor vanishes by 
Lemma~\ref{lem:ABsimp}. This leaves 
$$
\partial_k^2 \tilde{u} = X (\partial_k^2 Y)Z
$$
which simplifies to 
\begin{align*}
\partial^2_k\tilde{u}(\bzer) & = -p_d^2 \frac{\prod_{j<d}(1+p_j)}{1-p_d} \frac{\left(-2A'_k(\bp)B(\bpht{d})+2A(\bpht{d})B'_k(\bp)\right)}{B(\bpht{d})^2} \\
& = 2\frac{B(\bp)}{A(\bp)}\frac{\prod_{j<d}(1+p_j)}{1-p_d} \left[\frac{A'_kB(\bp) - A(\bp)B'_k(\bp)}{B(\bp)^2} \right] \\
& =\frac{2\prod_{j<d}(1+p_j)}{1-p_d} \left[\frac{A'_k(\bp)}{A(\bp)} - \frac{B'_k(\bp)}{B(\bp)} \right]
\end{align*}
by Lemma~\ref{lem:ABsimp}. 

Now consider $k=d$. Then since $X$ is independent of $\theta_d$, 
$\partial^2_d \tilde{u}$ evaluates at $\bt = \bzer$ to  $X \left[(\partial_d^2Y)Z + 2 (\partial_dY)(\partial_dZ)\right]$.
At this point we readily compute $\partial_dY = -2i, \partial_dZ = p_di/(1-p_d)^2, \partial^2_d Y = 4p_d^2 A/B = 4$. Thus 
$$
\partial_d^2 \tilde{u}(\bzer)  = \frac{4 \prod_{j<d}(1+p_j)}{(1-p_d)^2}.
$$
Thus
{\small 
$$
 -\frac{1}{2} \mathcal{H}(\tilde{u})(\bzer) = \frac{\overline{S}(\bp)}{4} \frac{\prod_{j<d} (1+p_j)}{1-p_d} \left[ \sum_{j=1}^{d-1} \frac{2}{p_jB_j(\bp)}\left(\frac{A'_j(\bp)}{A(\bp)} - \frac{B'_j(\bp)}{B(\bp)} \right) + \frac{4}{A(\bp)p_d(1-p_d)}\right].
$$}

We now compute the term $(-1/8) \mathcal{H}^2(\tilde{u}\tilde{g})(\bzer)$. The diagonal nature of the Hessian implies that $\mathcal{H}^2$ 
has the form $\sum_{j,k} c_j c_k \partial_j^2\partial_k^2$.  Lemma~\ref{lem:partialS} now allows further simplification, because it implies that each third partial derivative of the form $\partial_j^3\tilde{g}$ evaluates to zero. Thus, since $\tilde{u}$ vanishes to order at least $1$ at $0$, when we expand $\mathcal{H}^2(\tilde{u}\tilde{g})$ fully the only nonzero terms remaining on evaluation at $\bzer$ are of the form $\partial_d\tilde{u}\partial_d \partial^2_j \tilde{g}$. The coefficient of each such term is $4c_jc_d$.

It is easily computed that 
$$
\partial_d \tilde{u}(\bzer) = X(\bzer) \partial_dY(\bzer) Z(\bzer) =  -2i \frac{\prod_{j<d}(1+p_j)}{1-p_d}.
$$
By Lemma~\ref{lem:partialS},
$$
\partial_d\partial_j^2 \tilde{g}(\bzer) = -\partial_d\partial_j^2 \tilde{S}(\bzer)/\tilde{S}(\bzer) = \frac{2i p_j \left(p_d \tilde{A}'_j(\bzer) - p_d^{-1} \tilde{B}'_j(\bzer)\right)}{\overline{S}(\bp)}.
$$
Thus,
$$(-1/8) \mathcal{H}^2(\tilde{u}\tilde{g})(\bzer) = -\frac{\overline{S}(\bp)\prod_{j<d} (1+p_j)}{2p_d(1-p_d)A(\bp)}\sum_{j<d} \frac{\left(p_d \tilde{A}'_j(\bzer) - p_d^{-1} \tilde{B}'_j(\bzer)\right)}{B_j}$$
and hence, using a little more algebraic simplification (particularly the defining relation for $p_d$), we obtain
$$
C_\bp = \frac{\overline{S}(\bp)\prod_{j<d} (1+p_j)}{1-p_d} \left[ \frac{1}{A(\bp)p_d(1-p_d)} + \sum_{j=1}^{d-1} \frac{1-p_j}{2p_jB_j(\bp)}\left(\frac{A'_j(\bp)}{A(\bp)} - \frac{B'_j(\bp)}{B(\bp)} \right)\right].
$$
\end{proof}

\end{document}